\documentclass[12pt]{amsart}

\usepackage{graphicx} 

\usepackage{amsfonts,amssymb,eucal}
\usepackage{color}
\usepackage{amsthm} 
 \usepackage{amsmath} 
\usepackage{amscd}
 \usepackage{latexsym}
\usepackage{bbold,mathbbol,bbm}
\numberwithin{equation}{section}
\newcommand{\Ad}{\ensuremath{{\mbox{\rm{Ad}}}}}% Ad- representation

\newcommand{\e}{{\mbox{\rm e}}}

\newcommand{\mb}[1]{{\mbox{\boldmath{$#1$}}}}% mathematical bold
\newcommand{\mc}[1]{{\mathcal{#1}}}% mathematical caligraphic
\newcommand{\got}[1]{{\mathfrak{#1}}}% gothic with mbox for  mathematic
\newcommand{\db}[1]{{\mathbb{#1}}}% double

\newcommand{\pa}{\partial}
\newcommand{\R}{\ensuremath{\mathbb{R}}}
\newcommand{\C}{\ensuremath{\mathbb{C}}}
\newcommand{\N}{\ensuremath{\mathbb{N}}}

\newcommand{\g}{\ensuremath{\got{g}}}% Lie algebra g
\newcommand{\m}{\ensuremath{\got{m}}}%  algebra m

\newtheorem{Remark}{Remark}

\newtheorem{Theorem}{Theorem}

\newtheorem{Proposition}{Proposition}
\newtheorem{lemma}{Lemma}

\newtheorem{Comment}{Comment}
\theoremstyle{definition}

\newcommand{\SL}{\ensuremath{{\mbox{\rm{SL}}(2,\R)}}}
\newcommand{\Ka}{K\"ahler}
\newcommand{\mr}[1]{{\mathrm{#1}}}% mathematical roman
\def\ii{\operatorname{i}}
\newcommand{\dd}{\operatorname{d}}

\newcommand{\h}{\ensuremath{\got{h}}}% Lie algebra h
\renewcommand{\Re}{\operatorname{Re}}
\renewcommand{\Im}{\operatorname{Im}}
\allowdisplaybreaks

\begin{document}

\title{Remarks on the geometry of the  extended Siegel--Jacobi upper half-plane}
  \author{Elena Mirela   Babalic}
\address[Elena Mirela  Babalic]{National
 Institute for Physics and Nuclear Engineering\\
         Department of Theoretical Physics\\
         PO BOX MG-6, Bucharest-Magurele, Romania}

\email{mbabalic@theory.nipne.ro}

\author{Stefan  Berceanu}
\address[Stefan  Berceanu]{National
 Institute for Physics and Nuclear Engineering\\
         Department of Theoretical Physics\\
         PO BOX MG-6, Bucharest-Magurele, Romania}
       \email{Berceanu@theory.nipne.ro}
     
       \begin{abstract}
  The  real Jacobi group $G^J_1(\mathbb{R})={\rm
    SL}(2,\mathbb{R})\ltimes {\rm H}_1$, where  ${\rm H}_1$ denotes
  the 3-dimensional Heisenberg group, 
  is parametrized by the  $S$-coordinates
  $(x,y,\theta,p,q,\kappa)$. We show that
  the  parameter  $\eta$ that  appears passing from  Perelomov's un-normalized  coherent
  state vector  based on the Siegel--Jacobi disk $\mathcal{D}^J_1$ to the
  normalized one  
  is  $\eta=q+\rm{i} p$. 
 The two-parameter invariant metric on the Siegel--Jacobi upper
 half-plane
 $\mathcal{X}^J_1=\frac{G^J_1(\R)}{\rm{SO}(2)\times\mathbb{R}}$  is expressed in the variables
  $(x,y,\rm{Re}~\eta,\rm{Im}~\eta)$.
  It is proved that 
   the five dimensional manifold
   $\tilde{\mathcal{X}}^J_1=\frac{G^J_1(\R)}{\rm{SO}(2)}\approx
   \mathcal{X}^J_1\times \mathbb{R}$, called
   extended  Siegel--Jacobi upper half-plane, is a reductive,
   non-symmetric, non-naturally 
  reductive manifold with respect to the three-parameter metric invariant to
  the action of  $G^J_1(\mathbb{R})$, and  its
  geodesic vectors are determined.

\end{abstract}

 \maketitle

%%%%%%%\noindent\today
\tableofcontents
\section{Introduction}

The Jacobi group is defined as the semi-direct product of the
Heisenberg group and the symplectic group of appropriate dimension. The Jacobi group is intensively studied in Mathematics, Theoretical
and Mathematical Physics
\cite{jac1,SB15, SB19b,BERC08B,gem,bs,ez}, \cite{Y02}-\cite{Y10}.  We have
studied the  Jacobi group  $G^J_n:=\mr{H}_n\rtimes{\rm Sp}(n,\R)_{\C}$,
where  $\rm{H}_n$ denotes the $(2n+1)$-dimensional Heisenberg group
and ${\rm Sp}(n,\R)_{\C}:= {\rm Sp}(n,\C)\cap {\rm U}(n,n)$
\cite{sbj,nou}.

The real Jacobi group of degree $n$ is defined as $G^J_n(\R):={\rm
  Sp}(n,\R)\ltimes \mr{H}_n(\R)$, where  ${\rm Sp}(n,\R)_{\C}$ and $G^J_n$ are isomorphic to~${\rm Sp}(n,\R)$
and~$G^J_n(\R)$ respectively as real Lie groups, see
\cite[Proposition 2]{FC}, \cite{gem}, \cite{Y07}. To simplify the
notation we will denote in the following  $\mr{H}_n(\R)$ also with   $\mr{H}_n$.

The Siegel-Jacobi  ball  $\mc{D}^J_n$ is a
$G^J_n$-homogenous manifold, whose points are in
$\C^n\times\mc{D}_n$ \cite{sbj},  where  $\mc{D}_n\approx  \operatorname{Sp}(n, \R
)_{\C}/\operatorname{U}(n)$ denotes the Siegel (open)  ball  of degree $n$ \cite{helg}. 

The Jacobi group   is a unimodular, non-reductive, algebraic group of
Harish-Chandra type  \cite{gem,LEE03}, \cite{SA71}-\cite{SA80}, and 
 $\mc{D}^J_n$  is a  reductive,
non-symmetric manifold associated to the Jacobi group $G^J_n$ by the
generalized Harish-Chandra embedding \cite{SB15}, \cite{SB19b}. 
The holomorphic irreducible unitary representations of $G^J_n$ 
based on $\mc{D}^J_n$    constructed in 
 \cite{gem,bs}, \cite{tak}-\cite{TA99} 
are   relevant   to important areas of mathematics such as Jacobi forms, automorphic forms,
L-functions and modular forms, spherical functions, the ring of
invariant differential operators, 
theta functions, Hecke operators, Shimura varieties and  Kuga fiber varieties.

The Jacobi group was investigated by mathematicians
 \cite{lis2,lis,neeb96,neeb} in the context
of coherent states (CS) \cite{mosc,mv,perG}. 
Some CS systems based on $\mc{D}^J_n$  have been
considered   in the framework of quantum mechanics, geometric quantization,
dequantization, quantum optics, squeezed states, quantum
teleportation,  quantum tomography,  nuclear structure,  signal  processing, 
Vlasov kinetic equation  \cite{chi,gbt,KRSAR82,marmo,nish,Q90,SH03}.

The starting point in  Perelomov's approach to CS is  the triplet
$(G,\pi,\got{H})$, where $\pi$ is a unitary, irreducible representation
of the Lie
group $G$ on a separable complex  Hilbert space $\got{H}$  \cite{perG}. 

Two types of CS-vectors belonging to  $\got{H}$ are locally defined on
$M=G/H$:  the normalized (un-normalized) CS-vector
$\underline{e}_x$ (respectively, $e_z$) 
 \begin{equation}\label{2.1}
\underline{e}_x=\exp(\sum_{\phi\in\Delta^+}x_{\phi}{\mb{X}}^+_{\phi}-{\bar{x}}_{\phi}{\mb{X}}^-_{\phi})e_0,
\quad e_z=\exp(\sum_{\phi\in\Delta^+}z_{\phi}{\mb{X}}^+_{\phi})e_0,
\end{equation}
where $e_0$ is the extremal weight vector of the representation $\pi$,
$\Delta^+$ is the set of positive roots
of the Lie algebra $\got{g}$, and %%%%%5  $X_{\phi}$, $\phi\in\Delta$
$X^+_{\phi}$   ($X^-_{\phi}$) 
 are the positive (respectively, negative) generators. For  $X\in
 \got{g}$  we denoted in \eqref{2.1} $\mb{X}:=\dd\pi(X)$  \cite{SB03,SB14,perG}.

In the standard procedure of CS,
the  $G$-invariant \Ka~ two-form  on a $2n$-dimensional homogenous 
manifold $M=G/H$ is obtained from the \Ka~ potential $f$ via the recipe
\begin{subequations}\label{KALP}
  \begin{align}-\ii\omega_M & =\pa\bar{\pa}f, ~f(z,\bar{z})=\log
  K(z,\bar{z}), ~K(z,\bar{z}):=(e_{{z}},e_{{z}}),\label{KALP1}\\
\omega_M(z,\bar{z}) & =\ii \sum_{\alpha,\beta}h_{\alpha\bar{\beta}}\dd
  z_{\alpha}\wedge \dd \bar{z}_{\beta},~
  h_{\alpha\bar{\beta}}=\frac{\pa^2 f}{\pa z_{\alpha}\pa
      \bar{z}_{\beta}},~
                      h_{\alpha\bar{\beta}}=\bar{h}_{\beta\bar{\alpha}},~\alpha,\beta=1,\dots,n,\label{KALP2}
  \end{align}
  \end{subequations}
where   $K(z,\bar{z})$
is  the scalar product of two  un-normalized Perelomov's  CS-vectors $e_{{z}}$ at
$z\in M$  \cite{sbj,SB15, perG}.

It is well known, see \cite[Theorem
4.17]{ball}, \cite[Proposition 20]{SB19a}, \cite[(6), p 156]{kn}, 
that the condition $\dd \omega=0$ for a Hermitian
manifold to have a \Ka~ structure is equivalent with the conditions
\begin{equation}\label{EQK}
  \frac{\pa h_{\alpha\bar{\beta}}}{\pa z_{\gamma}}= \frac{\pa
      h_{\gamma\bar{\beta}}}{\pa z_{\alpha}}, \quad\text{or}\quad \frac{\pa h_{\alpha\bar{\beta}}}{\pa z_{\gamma}}= \frac{\pa
      h_{\alpha\bar{\gamma}}}{\pa z_{\bar{\beta}}},\quad\alpha,\beta,\gamma =1,\dots,n.
 \end{equation}

   In accord with  \cite[p 42 ]{ball}, \cite[p 28]{green},
   \cite[Appendix B]{SB19a},     the Riemannian metric associated with the Hermitian  metric
    on the manifold $M$ in local coordinates is 
    \begin{equation}\label{asm}\dd
    s^2_{M}(z,\bar{z})=\sum_{\alpha,\beta}h_{\alpha\bar{\beta}}\dd
    z_{\alpha}\otimes\dd \bar{z}_{\beta}.\end{equation}

Using the CS approach, in \cite{jac1} we have determined the \Ka~
invariant two-form $\omega_{\mc{D}^J_1}(w,z)$ on the Siegel--Jacobi disk
$\mc{D}^J_1=\frac{G^J_1}{\rm{U}(1)\times \R}\approx \mc{D}_1\times\C$,
  where  the Siegel disk $\mc{D}_1$ is realized as $\{ w\in\C |
  ~|w|<1\}$. In \cite{jac1,BER77,FC}  we applied the partial Cayley transform
  to  $\omega_{\mc{D}^J_1}(w,z)$ and we obtained
the \Ka~ invariant two-form on the Siegel--Jacobi upper half-plane 
$\mathcal{X}^J_1=\frac{G^J_1(\mathbb{R})}{{\rm SO}(2)
  \times\mathbb{R}} \approx\mathcal{X}_1 \times\mathbb{R}^2$, firstly
determined by  \Ka~and   Berndt \cite{mlad,bern84,bern,bs,cal3,cal}, where
$\mathcal{X}_1$ denotes the Siegel upper half-plane, realized as
$\{ v\in \C|
  \Im v>0\}$. The construction has been generalized in \cite{sbj,nou} for the
Jacobi group of degree $n$. 
In~\cite{SB15} we have underlined that the metric associated to the K\"ahler
two-form on the Siegel--Jacobi ball $\mc{D}^J_n=\frac{G^J_n}{\rm{U}(n)\times\R}$ is a  balanced metric \cite{arr,don,alo}.

In \cite{SB19b}    we introduced a  five-dimensional  manifold $\tilde{\mathcal{X}}^J_1=
\frac{G^J_1(\mathbb{R})}{{\rm
    SO}(2)}\approx\mathcal{X}_1\times\mathbb{R}^3$, called extended
Siegel--Jacobi upper half-plane. Because in Berezin's approach to CS
on $M=G/H$  traditionally  are considered $G$-homogenous \Ka~ metrics
on $M$, we 
were intersted in determining the invariant metrics as well   on
$\tilde{\mathcal{X}}^J_1$, 
so we had  to abandon  Berezin's  procedure to obtain  balanced 
 metric via the CS approach based on homogenous \Ka~
manifolds \cite{ber73}-\cite{berezin} and we applied in  \cite{SB19b}
 Cartan's moving frame method \cite{cart4,cart5,ev}, which allows to determine
 invariant metrics on  odd or even dimensional manifolds.

Mathematicians  consider the real Jacobi group $G^J_1(\R)$
as subgroup of  $\text{Sp}(2,\R)$. We  followed  this approach in
\cite{SB19b,BERC08B}, while  in 
\cite{jac1}-\cite{SB14}  the Jacobi group was
investigated via  the construction of Perelomov's CS.
We adopt  the notation from  \cite{bs,ez} for the real  Jacobi group   $G^J_1(\R)$,  
realized as  submatrices of $\text{Sp}(2,\R)$ of the form
\begin{equation}\label{SP2R}
g=\left(\begin{array}{cccc} a& 0&b &   q\\
\lambda &1& \mu & \kappa\\c & 0& d &  -p\\
         0& 0& 0& 1\end{array}\right),~ M=
    \left(\begin{array}{cc}a&b\\c&d\end{array}\right),~ \det M
   =1, \end{equation}
where
\begin{equation}\label{DEFY}Y:=(p,q)=XM^{-1}=(\lambda,\mu) \left(\begin{array}{cc}a&b\\c&d\end{array}\right)^{-1}=(\lambda d-\mu
  c,-\lambda b+\mu a)\end{equation} 
is related to the Heisenberg group $\rm{H}_1$ described by
$(\lambda,\mu,\kappa)$.  For  coordinatization of  the  real Jacobi
group  we adopt  the so called  $S$-coordinates
  $(x,y,\theta,p,q,\kappa)$  \cite{bs}.

  The present investigation is a continuation of \cite{SB19b}, where we have obtained invariant metrics for several
 homogeneous
 manifolds associated with  the real Jacobi group. In particular, we
 have determined the 2 (3)--parameter invariant metric on 
 $\mc{X}^J_1$,  (respectively, $\tilde{\mc{X}}^J_1$). We proved  in \cite{SB19b} that
   $\mc{X}^J_1$ is a non-symmetric, not naturally
 reductive space  with
 respect to  the balanced metric.

Below we motivate our interest for naturally reductive spaces.

We  denoted  by $FC$ \cite{nou} the change of variables $x \rightarrow  z$ in
formula \eqref{2.1} such that
\begin{equation}\label{2.2}
\underline{e}_x =  (e_z,e_z)^{-\frac{1}{2}}e_z; \quad z = FC(x).
\end{equation}
In  Remark 3 of the paper \cite{sbl},  devoted to coherent states with
support on Hermitian symmetric spaces, we observed that
\[
{\text{   For~symmetric manifolds the  FC-transform   gives
  geodesics\qquad (A)}}
\]
i.e. {\it for symmetric spaces} $M=G/H$, {\it the relation} $\exp(tz(x))=\exp (tFC(x))$ {\it gives
geodesics through   the identity of}
$M$. Assertion  (A) was
verified by direct calculation for the complex  Grassmann manifold
$G_n(\C^{m+n})= \frac{{\mr{SU}}(n+m)}{{\mr S}({\mr U}(n)\times {\mr
    U}(m))}$ and its noncompact dual
$\frac{{\mr{SU}}(n,m)}{{\mr S}({\mr U}(n)\times {\mr U}(m))}$
\cite{ber97A}.  Looking for  a geometric  mening of the
phase  of the scalar product of two un-normalized Perelomov's
CS-vectors \cite{ber97,SBS}, in \cite[Remark
1]{ber97} we showed that   assertion (A) is true for a larger class of
manifolds verifying a  technical condition which includes the naturally
reductive spaces, a natural generalization of symmetric spaces.  

So we have the following sequence of  space  inclusions 
\[
 { \text{Hermitian symmetric}}
\subset{\text{symmetric}}\subset{\text{naturally reductive}}
\subset{\text{g. o.}} 
\]
Indeed, in  \cite{BERA,BERB} we observed that the Hermitian
symmetric spaces are in particular naturally reductive spaces. Let 
$M=G/H$ be  a reductive Riemannian homogenous space \cite{nomizu}. We  have the direct 
sum of non-intersecting  vector spaces 
$\got{g}=\got{h}\oplus\got{m}$, 
and  geodesics on naturally reductive manifolds $M$ are obtained just
by taking the  exponential of
$\got{m}$ \cite{nomizu}.   We recall that the  g. o. spaces  are  Riemannian
homogeneous spaces $(M,g)$  with origin $p=\{H\}$ where all the
geodesics are  orbits of one parameter group of isometries $\exp tZ$,
$Z\in\got{m}$.  $X\in \got{g}\setminus \{0\}$ is a geodesic vector
if the curve  $\gamma(t)=\exp(tX)(p)$ is geodesic with respect to the
Riemannian connection \cite{kwv}.

In    \cite[Lemma 3]{SB05}, \cite[Lemma 6.11 and  Remark
6.12]{jac1} we proved that $(\C,\mc{D}_1)\ni (z,w)=FC(\eta,w)$, but in
\cite[Proposition 5.8]{SB19b} we showed  that $\mc{X}^J_1$ is not
naturally reductive with respect to the balanced metric.
Consequently,  the FC-transform on $\mc{D}^J_1$ does not generate  geodesics as in (A), see also \cite{GAB}.

The paper is laid out as follows. In  Section  \ref{UHF} we make  several changes of
coordinates in  the \Ka~ two-form $\omega_{\mc{D}^J_1}$ which are used in Section \ref{INV}. Section \ref{HS1}
(\ref{SLSR}), extracted from \cite{SB19b}, gives  information
on the embedding of $\rm{H}_1$ 
(respectively $\rm{SL}(2,\R)$)  in $\rm{Sp}(2,\R)$. Minimal
information on the real Jacobi group as a subgroup of $\rm{Sp}(2,\R)$ is
given  in Section \ref{JG1}. Lemma \ref{LEMN}, an enlarged and
improved version of  \cite[Lemma 5.1]{SB19b},  establishes the
action of the real Jacobi group on some of its   homogenous
spaces.  In  Proposition \ref{4pr}, an improved version of \cite[Proposition
5.2]{SB19b}, the fundamental vector fields (FVF) on homogenous
spaces associated to the Jacobi group are
calculated.  Comment
\ref{CM1},  Proposition   \ref{Pr5} and  Proposition \ref{BIGTH} enrich the
corresponding assertions  in \cite{SB19b}.  In the  last section,
where geometric properties of the five dimensional manifold 
$\tilde{\mc{X}}^J_1$ are investigated, also   are presented very sketchy the
definitions of the mathematical objects
investigated here, see
full details in 
 \cite{SB19b} and the 
extended version  \cite{SB19a}. The paper is concluded  with 
Comment \ref{CM2}.

 The new  results of the  present paper are contained in: Remark
 \ref{STR},  where we  emphasize that the S-variables $p,q\in\R$ have a ``CS-meaning''  given by
 the simple relation  $\eta =
 q+\ii p$;   item d) in
 Lemma \ref{LEMN};  item g) in Proposition \ref{4pr}; equation
  \eqref{METRS4} in Proposition \ref{Pr4}; Proposition \ref{Pr5}, which shows that the
 FVF determined in Proposition \ref{4pr} f) are Killing vectors for
 the invariant metric on  $\tilde{\mc{X}}^J_1$; Proposition \ref{BIGTH}, which
 shows that the FVF determined  in  Proposition \ref{4pr} g) are
 Killing vectors for
  the invariant metric of $G^J_1(\R)$.
  In  Theorem \ref{PRLST}, which summarises
   the main results of the present paper, we  show that
  $\tilde{\mc{X}}^J_1$ is a non-symmetric, non-naturally reductive
space with respect to the  three-parameter invariant metric. In the same
theorem  we 
determine also  the geodesic vectors on $\tilde{\mc{X}}^J_1$.

\textbf{Notation}

We denote by $\mathbb{R}$, $\mathbb{C}$, $\mathbb{Z}$ and $\mathbb{N}$ 
 the field of real numbers, the field of complex numbers,
the ring of integers,   and the set of non-negative integers, respectively. We denote the imaginary unit
$\sqrt{-1}$ by~$\ii$, the real and imaginary parts of a complex
number $z\in\C$ by $\Re z$ and $\Im z$ respectively, and the complex 
conjugate of $z$ by $\bar{z}$. We denote by~$|M|$ or by $\det(M)$ the determinant of matrix~$M$. $M(n,m,\db{F})$ denotes the set
of $n\times m$ matrices with entries in the field $\db{F}$. We denote by $M(n,\db{F})$ the set $M(n,n,\db{F})$. If $A\in
M(n,\db{F})$, then
$A^t$ denotes the transpose of~$A$. %%% $\un$~denotes the identity matrix of
                             %%% degree $n$.xx We denote by ``${\rm
                             %%% d}$'' the differential.
 We denote by ${\dd }$ the differential. 
We use Einstein convention i.e.  repeated indices are
implicitly summed over. The scalar product of vectors in the Hilbert
space  $\got{H}$ is
denoted $(\cdot,\cdot)$.  The set of vector fields (1-forms) is denoted
by $\got{D}^1$ (respectively $\got{D}_1$). If $\lambda\in\got{D}_1$ and $L\in\got{D}^1$, then
$\langle \lambda\,|\,L\rangle $ denotes their pairing. If~$X_i$,
$i=1,\dots,n$ are vectors in a  vector space $V$
over the field $\db{F}$, then $\langle X_1,X_2,\dots,X_n\rangle_{
\db{F}}$ denotes their span over $\db{F}$.  If we
denote with Roman  capital letteres the Lie  groups, then their
associated Lie algebras are denoted with the corresponding lowercase letter.

\section{Invariant \Ka~two-forms on the Siegel--Jacobi upper half-plane}\label{UHF}
The next proposition is an improved and enlarged version of
\cite[Proposition 2.1]{SB19b}. Below $(w,z)\in  ( \mc{D}_1,\C)$,
$(v,u)\in (\mc{X}_1,\C)$, and 
the parameters $k$ and $\nu$ come from representation theory of the
Jacobi group: $k$ indexes the positive discrete series of ${\rm
  SU}(1,1)$, $2k\in\N$, while $\nu>0$ indexes the representations of
the Heisenberg group \cite{jac1}.

\begin{Proposition}\label{PRFC}
a) Let us consider the \Ka~ two-form 
\begin{equation}\label{kk1}
  -\ii \omega_{\mc{D}^J_1}
(w,z)\!=\!\frac{2k}{P^2}\dd w\wedge\dd
  \bar{w}+\nu \frac{A\wedge\bar{A}}{P},~P:=1-|w|^2,~A=A(w,z):=\dd
  z+\bar{\eta}\dd w,
\end{equation}
$G^J_0$-invariant to the action on the Siegel--Jacobi disk $\mc{D}^J_1$
\begin{equation}\label{22}
  \emph{\text{SU}}(1,1)\times\C\ni(\left(\begin{array}{cc}\mc{P}&\mc{Q}\\\bar{\mc{Q}}&\bar{\mc{P}}\end{array}\right),\alpha)\times
(w,z)=(\frac{\mc{P}w+\mc{Q}}{\bar{\mc{Q}}w+\bar{\mc{P}}},\frac{z+\alpha-\bar{\alpha}w}{\bar{\mc{Q}}w+\bar{\mc{P}}}).\end{equation}

We have the change of variables $(w,z)\rightarrow (w,\eta)$
\begin{gather}\label{E32}
{\rm FC}\colon \
 z=\eta-w\bar{\eta},\qquad {\rm FC}^{-1}\colon \
 \eta=\frac{z+\bar{z}w}{P},
\end{gather}
and
\begin{equation}\label{E32a}
  {\rm FC}\colon \ A(w,z)\rightarrow {\rm d} \eta -w{\rm d}
  \bar{\eta}.
\end{equation}
The complex  two-form
\begin{equation}\label{mind}
  \omega_{\mc{D}^J_1}(w,\eta):=FC^*(\omega_{\mc{D}^J_1}(w,z))
\end{equation}
is not a \Ka~two-form.

The symplectic  form corresponding to  the  FC-transform applied to K\"ahler two-form \eqref{kk1} is invariant to the action
$(g,\alpha)\times (w ,\eta)= (w_1,\eta_1)$ of $G^J_0$ on
$\C\times\mc{D}_1$
\begin{equation}\label{etaA}\eta_1=\mc{P}(\eta+\alpha)+\mc{Q}
  (\bar{\eta}+\bar{\alpha}),\end{equation}
where $\mc{P},\mc{Q}$ appear in \eqref{22}.

b) Using the partial Cayley transform 
\begin{subequations}\label{210}
\begin{align}
\Phi^{-1}:& ~v=\ii \frac{1+w}{1-w},~~u=\frac{z}{1-w}, ~~w,z\in\C,~ |w|<1;\\\
\Phi: & ~w=\frac{v-\ii}{v+\ii},~~z=2\ii
        \frac{u}{v+\ii},~~v,u\in\C,~\Im v>0,
\end{align}
\end{subequations}
we obtain
\[ A\left(\frac{v - \ii}{v+ \ii},\frac{2\ii
      u}{v + \ii}\right)=\frac{2\ii}{v+\ii}B(v,u),\]
where \begin{equation}\label{BFR2}
  B(v,u) := {\rm d} u - \frac{u - \bar{u}}{ v- \bar{v}}{\rm d} v.\end{equation}
The \Ka~ two-form of Berndt--\Ka
\begin{equation}
- \ii \omega_{\mc{X}^J_1}(v,u) = -\frac{2k}{(\bar{v} - v)^2} \dd
v\wedge \dd\bar{v}+ \frac{2\nu}{\ii(\bar{v} - v)}B\wedge\bar{B}, \label{BFR}\\
\end{equation}
 is $G^J(\R)_0$-invariant to the action on the Siegel--Jacobi upper
 half-plane $\mc{X}^J_1$
 
 \begin{gather}\label{TSLL}
   \left(\SL\times\C^2\ni\left(\begin{matrix}a&b\\c&d\end{matrix}\right),\alpha\right)\times
(v,u)\!=\!\left(\frac{av+b}{cv+d},\frac{u+nv+m}{cv+d}\right), \alpha=
m+\ii n.
\end{gather}

 We have the   change of variables ${\rm FC}_1\colon \
 (v,u)\rightarrow (v,\eta)$
 \begin{equation}\label{FC1MIN}
{\rm FC}_1\colon \ 2\ii u=(v+\ii)\eta-(v-\ii)\bar{\eta}, \qquad {\rm FC}^{-1}_1 \colon \ \eta=\frac{u\bar{v}-\bar{u}v + \ii(\bar{u}-u)}{\bar{v}-v}.\end{equation}

c) If
\begin{equation}\label{UVPQ}
\C\ni   u:=pv+q,~~p,q\in\R, \quad \C\ni v:=x+\ii y, ~x,y\in\R,~ y>0,\end{equation}
then
\begin{equation}\label{BUVpq}
  B(v,u)=\dd u -p \dd v,\end{equation}
and\
\begin{equation}\label{BUVpq1}
  B(v,u)=B(x,y,p,q)=v\dd p+\dd q=(x+\ii y)\dd p +\dd q.
\end{equation}

d)  If we have  \eqref{FC1MIN},  \eqref{UVPQ} and we write
\begin{equation}\label{CNI}
  \C\ni \eta:= \chi +\ii \psi,~ \chi, \psi\in\R,
\end{equation}
then we get  the change of coordinates
\begin{equation}\label{CNI1} 
 (x,y,p,q)\rightarrow (x,y,\chi,\psi): ~ \psi=p,~\chi=q,
\end{equation}
and 
 \begin{equation}\label{CNI11}B(v,u)=B(x,y,\chi,\psi)= x\dd
   \psi+\dd \chi + \ii y \dd \psi.
 \end{equation}
 We  also have the relations
 \begin{equation}\label{surp}
   \eta=  q+\ii p,~ q=\frac{1}{2}(\eta+\bar{\eta}),~p=\frac{1}{2\ii}(\eta-\bar{\eta}).
   \end{equation}

Given  \eqref{UVPQ} and
\begin{equation}\label{uxieta}
  \C\ni u:=\xi+\ii \rho, \quad \xi,\rho\in\R,\end{equation}
 we obtain  the change of variables
\begin{equation}\label{LASt1}
  (x,y,\xi,\rho)\rightarrow (x,y,p,q): ~\xi=px+q,~\rho=py,
\end{equation}
and 
\begin{equation}\label{LAST12}
  B(v,u)=\dd u-\frac{\rho}{y}\dd v=\dd (\xi+\ii \rho)-\frac{\rho}{y}\dd (x+\ii y).
 \end{equation}

If we have \eqref{UVPQ} and \eqref{CNI}, then, with \eqref{FC1MIN},  we have the change of
coordinates
\begin{equation}\label{LASt}
  (x,y,\xi,\rho)\rightarrow (x,y,\chi,\psi):~~\xi=\psi x+\chi,~~\rho=\psi
  y.
  \end{equation}

\end{Proposition}
\begin{proof}
 a)  We  determined  in \cite{jac1,SB14} the scalar
product  $K(w,z):=(e_{wz},e_{wz}$) 
of two Perelomov's CS states based on the Siegel--Jacobi disk.
The associated \Ka~ potential on $\mc{D}^J_1$ is
 \begin{equation}\label{FWZ}f(w,z)=-2k\log(P)+\nu\frac{2|z|^2+\bar{w}z^2+w\bar{z}^2}{2P}.\end{equation}
 In \cite{jac1}  we  applied \eqref{KALP} to the potential \eqref{FWZ}
 and we obtained the \Ka~two-form on $\mc{D}^J_1$
   \begin{equation}\label{FFCX}-\ii \omega_{\mc{D}^J_1}(w,z)=f_{z\bar{z}}\dd  z\wedge \dd
     \bar{z}+ f_{z\bar{w}}\dd z\wedge \dd \bar{w}-\bar{f}_{z\bar{w}}\dd
     \bar{z}\wedge\dd w+f_{w\bar{w}}\dd w\wedge \dd \bar{w}.\end{equation}

     The matrix corresponding  to the metric associated with the   \Ka~two-form \eqref{FFCX}
reads \cite[(5.11)]{SB14}
     \begin{equation}\label{kma}
       h(w,z)
       =\left(\begin{array}{cc}f_{z\bar{z}}&f_{z\bar{w}}\\\bar{f}_{z\bar{w}}&f_{w\bar{w}}\end{array}\right)
       =\left(\begin{array}{cc}\frac{\nu}{P}&\nu\frac{\eta}{P}\\\nu\frac{\bar{\eta}}{P}&
     \frac{2k}{P^2}+\nu\frac{|\eta|^2}{P}\end{array}\right).
       \end{equation}
It is easy to verify that the matrix elements of the matrix \eqref{kma} verify
the conditions \eqref{EQK} and the Jacobi disk $\mc{D}^J_1$ is a
\Ka~manifold.

If we apply to the \Ka~two-form \eqref{kk1} the non-holomorphic
$FC$-transform, then the complex two-form \eqref{mind}    in the
variables $(w,\eta)$ is
not a \Ka~two-form \cite[Proposition 2, p 50]{AWeil}. This fact can be
directly verified:  if we introduce in \eqref{kma} the  value of
$\eta=\eta(w,z)$ given in \eqref{E32}, then the conditions \eqref{EQK}
for a complex two-form to  be a fundamental two-form are not satisfied.  

For the invariance \eqref{etaA} see \cite[(6.4)]{nou}.

     b)  
The \Ka~two-form on the Siegel--Jacobi upper half-plane   $\mc{X}^J_1$
was determined from 
$\omega_{\mc{D}^J_1}(w,z)$ using the partial Cayley  transform in  \cite{jac1,SB14,SB15}.
%%%%%%%%%%%
       Note that in the Berndt--\Ka~ approach in \cite{bern84} the \Ka~
       potential \eqref{POT} is just ``guessed'', see Comment \ref{CM1}.

      We have $$\frac{A\wedge\bar{A}}{P}=\frac{1}{y} B\wedge\bar{B}.$$

 For \eqref{FC1MIN} see \cite[(3.4)]{SB14}.
  \end{proof}

Correlating \eqref{surp} in Proposition \ref{PRFC} with  the  \cite[Comment
6.12]{jac1},  (2.1),
(2.2) and  \cite[Lemma 2]{SB14}, we make the following surprising
remark giving a
``CS -  meaning'' to the  $S$-parameters $p,q$.
\begin{Remark}\label{STR}
  The FC-transform \eqref{E32} relates Perelomov's  un-normalized CS-vector
  $e_{w,z}$ with the normalized one $\underline{e}_{w,\eta}$
 \[
    \underline{e}_{w\eta }=(e_{wz}, e_{wz})^{-\frac{1}{2}}e_{wz},\quad
    w\in\mc{D}_1,~z,\eta \in \C,
  \]
  
and the S-variables $p,q$ are related to the parameter $\eta$ by the
simple relation
\begin{equation}\label{strange}
    \eta=q +\ii p.
  \end{equation}
 \end{Remark}

\section{The Heisenberg group embedded in  $\rm{Sp}(2,\R)$}\label{HS1}
In this section, extracted from \cite[Section ~3]{SB19b}, we
summarize  the
parametrization of the Heisenberg group used in \cite{bs}.

The composition law of the 3-dimensional  Heisenberg group $\rm{H}_1$ in \eqref{SP2R} is
\[
(\lambda,\mu,\kappa)(\lambda',\mu',\kappa')=(\lambda+\lambda',\mu+\mu',\kappa+\kappa'+\lambda\mu'-\lambda'\mu).
\]

As in
\eqref{SP2R} with $M=\mathbb{1}_2$,  we denote an element of $\rm{H}_1$ embedded in $\text{Sp}(2,\R)$ by
\begin{equation}\label{Real2} \rm{H}_1\ni g =  \left(\begin{array}{cccc}
 1& 0&  0 &   \mu\\
\lambda &1& \mu & \kappa\\
0 & 0& 1 &  -\lambda\\
         0& 0& 0& 1
\end{array}\right),~ g^{-1} =  \left(\begin{array}{cccc} 1& 0& 0&   -\mu\\
-\lambda &1& -\mu & -\kappa\\
0 & 0& 1 &  \lambda\\
         0& 0& 0& 1\end{array}\right).\end{equation}

A basis of the Lie algebra $\got{h}_1=<P,Q,R>_{\R}$  of the
Heisenberg group  $\rm{H}_1$ in the realization 
\eqref{Real2} in the space $M(4,\R)$ consists of the matrices
\[  %%%%%%%\label{PQR2}
P=\left(\begin{array}{cccc} 
0&0&0&0\\
1&0&0&0\\
0&0&0&-1\\
0&0&0&0\end{array}\right),
~Q=\left(\begin{array}{cccc} 
0&0&0&1\\
0&0&1&0\\
0&0&0&0\\
0&0&0&0\end{array}\right),~R=\left(\begin{array}{cccc} 
0&0&0&0\\
0&0&0&1\\
0&0&0&0\\
0&0&0&0\end{array}\right),
\]
which verify the commutation relations
\begin{equation}
\label{PQT1}[P,Q]=2R,~[P,R]=[Q,R]=0.
\end{equation}

If we write $$\rm{H}_1\ni g(\lambda,\mu,\kappa)=\mathbb{1}_4+\lambda P+\mu
Q+\kappa R,$$ then, using  the formulas \eqref{GMGP}, see details in
\cite[Section 3]{SB19a}
\begin{equation}\label{GMGP}
  g^{-1}\dd g=P\lambda^p+Q\lambda^q+R\lambda^r,~~\dd g
  g^{-1}=P\rho^p+Q\rho^q+R\rho^r,
  \end{equation}
we find the left-invariant one-forms and vector fields
\[
\left\{\begin{array}{l}
\lambda^p  = \dd \lambda\\ 
\lambda^q  = \dd {\mu}\\ 
\lambda^r  = \dd {\kappa}- \lambda\dd {\mu} +{\mu}\dd \lambda \end{array}\right.;  
{\mbox{~~~~~~~~~~~~~~}}\left\{\begin{array}{l} L^p=\pa _{\lambda}
                                -{\mu}\pa_ {\kappa}\\
           L^q= \pa_{\mu}+\lambda \pa_{\kappa}\\ L^r=\pa_{\kappa}
                            \end{array} \right. .
\]

\section{The $\SL$ group embedded in  $\rm{Sp}(2,\R)$}\label{SLSR}
In this section we extract from \cite[Section 4]{SB19b} the minimum
information we need  to understand the embedding of $\SL$ in the
4-dimensional matrix realization of $\rm{Sp}(2,\R)$.

An element $M\in \SL$ and its inverse are  realized as elements of  
 $\text{Sp}(2,\R)$ by the relations
\begin{equation}\label{ALOS}
M=
\left(\begin{array}{cc}a&b\\c&d\end{array}\right)\!\mapsto~
g\!=\!\left(\!\begin{array}{cccc} a& 0&b
          &0\\0&1&0&0\\c&0&d&0\\0&0&0&1\end{array}\!\right)\!\in G^J_1(\R),~ 
 g^{-1}\!=\!\left(\begin{array}{cccc} d& 0&-b
                &0\\0&1&0&0\\-c&0&a&0\\0&0&0&1\end{array}\right).  \end{equation}
A basis of the Lie algebra $\got{sl}(2,\R)=<F,G,H>_{\R}$ consists of
the matrices in $M(4,\R)$
\[
  F= \left(\begin{array}{cccc} 0& 0&1
          &0\\0&0&0&0\\0&0&0&0\\0&0&0&0\end{array}\right), ~G= \left(\begin{array}{cccc} 0& 0&0
          &0\\0&0&0&0\\1&0&0&0\\0&0&0&0\end{array}\right),~H= \left(\begin{array}{cccc} 1& 0&0
                                                                      &0\\0&0&0&0\\0&0&-1&0\\0&0&0&0\end{array}\right) .
\]
    $F,~G,~H$ verify the commutation relations \eqref{FGHCOM}
    \begin{equation}\label{FGHCOM}
[F,G]=H,~[G,H]=2G,~ [H,F]=2F.
\end{equation}

    With the
representation \eqref{ALOS}, we have
\[
  g^{-1}\dd g=F\lambda^f+G\lambda^g+H\lambda^h,~\dd
  g g^{-1}=F\rho^f+G\rho^g+H\rho^h.
\]
Using the parameterization \eqref{ALOS} for $\SL$, we find 
\[  %%%%%%%%\label{LEFTRIGHT}
\left\{\begin{array}{l}
\lambda^f  = d\dd b-b\dd d\\ 
\lambda^g  = -c \dd a + a\dd c\\ 
\lambda^h  = d \dd a - b \dd c=c\dd b-a \dd d\end{array}\right.;  
{\mbox{~~~~~~~~~~~~~~}}\left\{\begin{array}{l} \rho^f=-b \dd a +a\dd b\\
           \rho^g= d \dd c - c\dd d\\ \rho^h=d \dd a- c\dd b
\end{array} \right. . \]         %%%%%%%%%\end{equation}

 The Iwasawa       decomposition $M=NAK$
of  an element $M$ as in  \eqref{ALOS} reads
\[      %%%%%%%\label{MNAK}
  M=\left(\begin{array}{cc}1&x\\0&
                                   1\end{array}\right)
\left(\begin{array}{cc}y^{\frac{1}{2}}&
           0\\0& y^{-\frac{1}{2}}
          \end{array}\right)
\left(\begin{array}{cc}
\cos\theta &\sin\theta\\-\sin\theta
           &\cos\theta \end{array}\right),~y>0.
     \]    %%%%%%%%\end{equation}
     We find
     \begin{subequations}\label{SCXYT}
\begin{align}
a &=  y^{1/2}\cos\theta - xy^{-1/2}\sin\theta,\quad
b = y^{1/2}\sin\theta + xy^{-1/2}\cos\theta,\\
c &= -y^{-1/2}\sin\theta,\quad
d  = y^{-1/2}\cos\theta,
\end{align}
\end{subequations}
and 
\begin{equation}\label{SCINV}
x=\frac{ac+bd}{d^2+c^2}, ~ y=\frac{1}{d^2+c^2},~\sin\theta
=-\frac{c}{\sqrt{c^2+d^2}},~\cos\theta =\frac{d}{\sqrt{c^2+d^2}}. 
\end{equation}    
We determined in \cite{SB19b} the left-invariant vector fields $L^f,L^g,L^h$ on
$\SL$, dual  orthogonal  
to the left-invariant  one-forms  $\lambda^f,\lambda^g,\lambda^h$.
We  introduced the left-invariant one-forms
\begin{equation}\label{1l2l3l}
\lambda_1:=\sqrt{\alpha}(\lambda^f+\lambda^g), ~~~
~\lambda_2:=2\sqrt{\alpha}\lambda^h,~~~
\lambda_3:=\sqrt{\beta}(\lambda^f-\lambda^g). 
\end{equation}
In \cite{SB19b} we determined the  left-invariant vector fields $L^j$ such
that  $<\lambda_i|L^j>=\delta_{ij},~i,j=1,2,3$, where
\begin{equation}\label{1L2L3L}
L^1  :=\frac{1}{2\sqrt{\alpha}}(L^f+L^g),~~~
L^2   :=\frac{1}{2\sqrt{\alpha}}L^h,~~~
L^3  :=\frac{1}{2\sqrt{\beta}}(L^f-L^g).
\end{equation}
\section{The Jacobi group  $G^J_1(\R)$ embedded in  $\rm{Sp}(2,\R)$}\label{JG1}
\subsection{The composition law}

The real Jacobi group of index one is the semidirect product
of the real three-dimensional Heisenberg group $\rm{H}_1$ with
$\text{SL}(2,\R)$. The Lie algebra  of the  Jacobi
group $G^J_1(\R)$ is
given by
$\got{g}^J_1(\R)=<P,Q,R,F,G,H>_{\R}$, where the first three generators
$P,Q,R$ of $\got{h}_1$ verify the commutation relations
\eqref{PQT1}, the generators  $F,G,H$ of $\got{sl}(2,\R)$ verify 
 the commutation relations \eqref{FGHCOM} and the ideal $\got{h}_1$ in
 $\got{g}^J_1(\R)$ is determined by the non-zero commutation relations
\begin{equation}\label{MORCOM}
[P,F]=Q,~[Q,G]=P,~[P,H]=P,~ [H,Q]=Q. 
\end{equation}

Let $ g:= (M,h)\in G^J_1(\R) $, where $M$ is as in \eqref{ALOS}, while
$ h:=(X,\kappa)\in\rm{H}_1,~ X:=(\lambda,\mu)\in\R^2$ and similarly for
$g':=(M',h')$.
 The composition law of $G^J_1(\R)$ is
\begin{equation}\label{COMPL}
gg'=g_1,
\text{~where~}M_1=MM',~X_1=XM'+X',~\kappa_1=\kappa+\kappa'+\left|\begin{array}{c}XM'\\X'\end{array}\right|,
\end{equation} 
where
\begin{subequations}
\begin{align*}
g_1&=\left(\begin{array}{cc}aa'+bc'& ab'+bd'\\ ca'+dc'&
                                                        cb'+dd'\end{array}\right),\\
(\lambda_1,\mu_1) &= (\lambda'+\lambda a'+\mu c',\mu'+\lambda b'+\mu
                    d'),\\
\kappa_1 &= \kappa+\kappa'+\lambda q'-\mu p'.
\end{align*}
\end{subequations}
The inverse element  of  $g\in G^J_1(\R)$ is given by \begin{equation}\label{INVV}
(M,X,\kappa)^{-1}=(M^{-1},-Y,-\kappa)\rightarrow~g^{-1}=\left(\begin{array}{cccc}
                                                        d& 0& -b&
                                                                   -\mu\\-p
                                                          &1&-q
                                                                 &-\kappa\\-c&0&a&\lambda\\ 0&0&0&1\end{array}\right),
\end{equation}
where $Y$ was defined in \eqref{DEFY} and similarly for $Y'$, while $g$ has the general form given in \eqref{SP2R}.

Using the notation in \cite[p 9]{bs}, the {\it EZ-coordinates}
(EZ - from Eichler \& Zagier) of an element $g\in G^J_1(\R)$ as in \eqref{SP2R} are
$(x,y,\theta,\lambda,\mu,\kappa)$, where $M$ is related with
$(x,y,\theta)$ by \eqref{SCXYT}, \eqref{SCINV}.

The {\it S-coordinates} (S - from Siegel)  of $g=(M,h)\in G^J_1(\R)$ are
$(x,y,\theta,p,q,\kappa)$, where $(x,y,\theta)$ are expressed as
functions of $M\in \SL$  by \eqref{SCXYT}, \eqref{SCINV}.

\subsection{The action}
Let 
\begin{equation}\label{TAUZ}
\C\ni \tau:=x+\ii y,~~~
\C\ni  z:=p\tau+q=\xi+\ii \rho,\quad ~x,y,p,q, \xi,\rho\in\R. \end{equation}
 Let
$\mc{X}^J_1\approx \mc{X}_1\times\R^2$ be the Siegel--Jacobi upper half-plane,
where $\mc{X}_1=\{\tau\in\C| ~y:=\Im \tau>0\}$ is the Siegel upper half-plane,  and
$\tilde{\mc{X}}^J_1\approx\mc{X}^J_1\times\R$ denotes  the extended Siegel--Jacobi upper half-plane. Simultaneously with
the Jacobi group $G^J_1(\R)$ consisting of elements $(M,X,\kappa)$, we
considered the group $G^J(\R)_0$  of elements $(M,X)$ \cite{jac1,SB19b}.
Then: 
\begin{lemma}\label{LEMN} a) The action  $G^J(\R)_0\times \mc{X}^J_1\rightarrow
  \mc{X}^J_1$
is given by \begin{equation}\label{AC1}
(M,X)\times
(\tau',z')=(\tau_1,z_1),\emph{\text{~where~}}\tau_1=\frac{a\tau'+b}{c\tau'+d},~z_1=\frac{z'+n\tau'+m}{c\tau'+d}.\end{equation}

b) If  $z'=p'\tau'+q'$, $\tau'=x'+\ii y'$ as in \eqref{TAUZ}, then the
action 
\begin{equation}\label{AC11}
(M,X)\times (x',y',p',q')=(x_1,y_1,p_1,q_1)
\end{equation}
is given by the formula
\begin{equation}\label{AC12}
(p_1,q_1)=(p,q)+(p',q')
\left(\begin{array}{cc}a & b\\c & d\end{array}\right)^{-1}=(p+dp'-cq',q-bp'+aq').
\end{equation}

c)  The action $G^J_1(\R)\times \tilde{\mc{X}}^J_1\rightarrow
  \tilde{\mc{X}}^J_1$ is given by 
\begin{equation}\label{AC2}
\begin{split}
& (M,X,\kappa)\times
 (\tau',z',\kappa')  =(\tau_1,z_1,\kappa_1),\\
& (M,X,\kappa)\times (x',y',p',q',\kappa')  =(x_1,y_1,p_1,q_1,\kappa_1),\\
 & \kappa_1  =\kappa +\kappa' +\lambda
q'-\mu p',~
(p',q')= (\frac{\rho'}{y'},\xi'-\frac{x'}{y'}\rho'),~
(\lambda,\mu)=(p,q)M.
\end{split}
\end{equation}
d) The action $G^J_1(\R)\times G^J_1(\R)\rightarrow G^J_1(\R)$
corresponding to the composition law \eqref{COMPL} is 
\[     %%%%%%%%%%\begin{equation}\label{AC3}
 (M,X,\kappa)\times
 (x',y',\theta',p',q',\kappa')=(x_1,y_1,\theta_1,p_1,q_1,\kappa_1).
\]  %%%%%%%%%%\end{equation}
\end{lemma}

\subsection{Fundamental vector fields}

 We recall  the notion of FVF, see 
 \cite[Appendix A.1]{SB19b}, \cite[p 122]{helg} and   \cite[p
 51]{kn1}.
 
Let $M=G/H$ be  a homogeneous $n$-dimensional manifold and let us
suppose that  the group  $G$
acts transitively on $M$ from the {\it  left} , $G\times M\rightarrow M:~ g\times
x=y$, where $y=(y_1,\dots,y_n)^t$.
Then $g(t)\times
x =y(t)$, where $g(t)=\exp(tX)$,  $t\in\R$, generates  a curve  $y(t)$
in $M$ with
$y(0)=x$ and $\dot{y}(0)=X$. The
{\it fundamental vector field} associated  to $X\in\got{g}$ at $x\in M$ is defined as 
\[       %%%%%%%%\begin{equation}\label{XF1}
    X^*_x:=\frac{\dd}{\dd t}y(t)|_{t=0}=\frac{\dd}{\dd t}(\exp(tX)\times
x)|_{t=0}=\sum_{i=1}^n(X^*_i)_x\frac{\pa }{\pa z_i}, ~~(X^*_i)_x=\frac{\dd
  y_i(t)}{\dd t}|_{t=0}.
\]      %%%%%%%%5\end{equation}

With  the  action given in Lemma \ref{LEMN}, we get the  FVF on some homogenous spaces
associated to  the  real Jacobi group, see Proposition 5.2 in \cite{SB19b},
reproduced  below. Only the item g) is new. 
\begin{Proposition}\label{4pr}
a) The FVF expressed in the coordinates $(\tau,z)$
of the Siegel--Jacobi upper half-plane $\mc{X}^J_1$ on which the
reduced  Jacobi group $G^J(\R)_0$ acts
  by \eqref{AC1} are given by the
holomorphic vector fields
\begin{subequations}\label{EQQ1}
\begin{align}
F^* & =\pa_{\tau}, ~ G^*= -\tau^2\pa_{\tau} -z\tau\pa_{z},~H^*=2\tau\pa_{\tau}+z\pa_z;\\ 
P^* &  =\tau\pa_z,~Q^*=\pa_z, ~R^* =0. 
\end{align}
\end{subequations}

b) The   real holomorphic FVF corresponding to
$\tau:=x+\ii y, ~y>0 $, $z:=\xi+\ii \rho$  in the variables
$(x,y,\xi,\rho)$ are
\begin{subequations}\label{EQQ2}
\begin{align}
F^* & =F^*_1, ~ G^*= G^*_1 +(\rho y-\xi
      x)\pa_{\xi}-(\xi y+ x\rho)\pa{_\rho};\\ H^* & = H^*_1
     +\xi\pa_{\xi}+\rho\pa_{\rho}, ~
 P^*  =x\pa_{\xi}+y\pa_{\rho},~Q^*=\pa_{\xi},     ~R^* =0 , 
\end{align}
\end{subequations}
where $F^*_1,G^*_1,H^*_1$ are the FVF
\eqref{FUNDFGH} of 
the homogenous manifold $\mc{X}_1$
\begin{equation}\label{FUNDFGH}
F^*_1=\frac{\pa
}{\pa x},~ G^*_1=(y^2-x^2)\frac{\pa }{\pa x}-2 xy\frac{\pa}{\pa y},~
H^*_1=2(x\frac{\pa}{\pa x}+y\frac{\pa }{\pa y})
\end{equation} associated 
to  the generators $F,G,H$ of $\got{sl}(2,R)$ corresponding to the
action  \eqref{AC1} of $\rm{SL(2},\R)$ on $\mc{X}_1$.

c) If we express the FVF  in the variables
$(x,y,p,q)$, where $\xi=px+q$, $\rho= p y$, we find 
  \begin{subequations}\label{EQQ3}
\begin{align}
F^* & =F^*_1 -p\pa_q, ~ G^*= G^*_1-q\pa_p, ~
       H^*=H^*_1-p\pa_p+q\pa_q;\\
P^*  & =\pa_p,~Q^*=\pa_{q}, ~R^* =0. 
\end{align}
\end{subequations}
d) If  we consider the action \eqref{AC2} of $G^J_1(\R)$ on the points
$(\tau,z,\kappa)$ of $\tilde{\mc{X}}^J_1$, we get 
instead of \eqref{EQQ1} the FVF  in
   the variables $(\tau,z,p,q,\kappa)$ 
\begin{subequations}\label{NEWPQR1}
\begin{align}
F^* & =\pa_{\tau}, ~ G^*= -\tau^2\pa_{\tau} -z\tau\pa_{z},~H^*=2\tau\pa_{\tau}+z\pa_z;\\ 
P^* &  =\tau\pa_z+q\pa_{\kappa},~Q^*=\pa_z-p\pa_{\kappa},~R^*
      =\pa_{\kappa}, p=\frac{\Im (z)}{\Im (\tau)},~q=\frac{\Im(\bar{z}\tau)}{\Im(\tau)}. 
\end{align}
\end{subequations}
e) Instead of \eqref{EQQ2}, we get  the FVF in
$\tilde{\mc{X}}^J_1$ in the variables
$(x,y,\xi,\rho,\kappa)$
\begin{subequations}\label{NEWPQR2}
\begin{align}
F^* & =F^*_1, ~ G^*= G^*_1 +(\rho y-\xi
      x)\pa_{\xi}-(\xi y+ x\rho)\pa{_\rho};\\ H^* & = H^*_1
     +\xi\pa_{\xi}+\rho\pa_{\rho}, ~
 P^*  =x\pa_{\xi}+y\pa_{\rho}+q\pa_{\kappa},~Q^*=\pa_{\xi} -p\pa_{\kappa},     ~R^* =\pa_{\kappa} .
\end{align}
\end{subequations}
f) Instead of \eqref{EQQ3}, we get the FVF in the variables
$(x,y,p,q,\kappa)$ 
\begin{subequations}\label{NEWPQR3}
\begin{align}
F^* & =F^*_1 -p\pa_q, ~ G^*= G^*_1-q\pa_p, ~
       H^*=H^*_1-p\pa_p+q\pa_q;\label{EQQ31x}\\
P^* & =\pa_p+q\pa_k,~Q^*=\pa_q-p\pa_k,~R^*=\pa_{\kappa}.
\end{align}
\end{subequations}
g) The FVF in the S-variables corresponding to the
composition law \eqref{COMPL} of the Jacobi group $G^J_1(\R)$ are the
same as in \eqref{NEWPQR3}, except  $G^*$ 
\begin{subequations}\label{NEWPQR4}
\begin{align}
F^* & =F^*_1 -p\pa_q, ~ G^*= G^*_1-y\frac{\pa}{\pa \theta}, ~
       H^*=H^*_1-p\pa_p+q\pa_q;\label{EQQ31xx}\\
P^* & =\pa_p+q\pa_k,~Q^*=\pa_q-p\pa_k,~R^*=\pa_{\kappa}.
\end{align}
\end{subequations}
\end{Proposition} 

\section{Invariant metrics}\label{INV}

Using the formula
\[ 
  %%%%%%%%% \begin{equation}\label{INVEX}
 g^{-1}\dd g=
 \lambda^FF+\lambda^GG+\lambda^HH+\lambda^PP+\lambda^QQ+\lambda^RR,\]%%%%%%%\end{equation}
where $g$ is as in \eqref{SP2R} and  $g^{-1}$ as in \eqref{INVV}, 
we  calculated in \cite{SB19b} the left-invariant one-forms on $G^J_1(\R)$. 

The left-invariant vector fields $L^{\alpha}$ for the
real Jacobi group  $G^J_1(\R)$ are  orthogonal with respect to the
invariant one-forms $\lambda^{\beta}$,
 $$<\lambda^{\beta}|L^{\alpha}>=\delta_{\alpha\beta},
 ~~\alpha,\beta= F,G,H,P,Q,R.$$
 
Besides the formulas for $\lambda_1,\lambda_2,\lambda_3$ defined in
\eqref{1l2l3l}, we have  introduced in \cite{SB19b} the left-invariant one-forms
 \[  %%%%%%%%%\begin{equation}\label{4lf5lf6lf}
\lambda_4:= \sqrt{\gamma}\lambda^P,~\lambda_5:=\sqrt{\gamma}\lambda^Q,~\lambda_6:=\sqrt{\delta}\lambda^R,
\]      %%%%%%%\end{equation}
where $\lambda^P,\lambda^Q,\lambda^R$ are defined  in \cite[p 18]{SB19b}.
Also, besides the left-invariant vector fields $L^1,L^2,L^3$ defined 
in \eqref{1L2L3L}, we have 
introduced the left-invariant one forms
\[    %%%%%%%%\begin{equation}\label{4L5L6L}
L^4:=\frac{1}{\sqrt{\gamma}}L^P,~L^5:=\frac{1}{\sqrt{\gamma}}L^Q,~L^6:=\frac{1}{\sqrt{\delta}}L^R,
\] %%%%%%%%\end{equation}
where $L^P,L^Q,L^R$ are defined in \cite[Proposition 5.3]{SB19b}. 

The vector fields  $L^i$, $i=1,\dots,6$ verify the   commutations
relations
\begin{subequations}\label{FUFUFU}
\begin{align}
[L^1,L^2]& =-\frac{\sqrt{\beta}}{\alpha}L^3 & [L^2,L^3]&=\frac{1}{2\sqrt{\beta}}L^1 & [L^3,L^1] &=\frac{1}{\sqrt{\beta}}L^2\\
[L^1,L^4]& = -\frac{1}{2\sqrt{\alpha}}L^5 & [L^1,L^5]&=-\frac{1}{2\sqrt{\alpha}}L^4
 & [L^1,L^6 ]&= 0 \\
[L^2,L^4]&= -\frac{1}{2\sqrt{\alpha}}L^4& [L^2,L^5 ] &=
\frac{1}{2\sqrt{\alpha}}L^5 & [L^2,L^ 6]&=0  \\
[L^3,L^4]&= - \frac{1}{2\sqrt{\alpha}}L^5 & [L^3,L^5] &=\frac{1}{2\sqrt{\beta}}L^4 & [L^3,L^6]&= 0 \\
[L^4,L^5] & = \frac{2\sqrt{\delta}}{\gamma}L^6& [L^4,L^6]&= 0&
                                                               [L^5,L^6]&= 0.  
\end{align}
\end{subequations}

Following \cite[Proposition 5.4]{SB19b}, we obtain from the \Ka~ two-form \eqref{BFR} in 
 Proposition \ref{PRFC}  the metric on $\mc{X}^J_1$ in the convention
 of \eqref{asm},   replacing $v\rightarrow
  \tau, ~u \rightarrow z,~k\rightarrow 2c_1,~\nu\rightarrow \frac{c_2}{2}$. 

 Only \eqref{METRS4} in the
next proposition is new.

\begin{Proposition}\label{Pr4}
The two-parameter  balanced  metric  on the
  Siegel--Jacobi upper half-plane $\mc{X}^J_1$, left-invariant to the
  action \eqref{AC1}, \eqref{AC11} and  \eqref{AC12},  respectively \eqref{AC2} of
 the  reduced group  $G^J(\R)_0$,  is given by the formulas
\begin{subequations}\label{METRS}
\begin{align}\dd s_{\mc{X}^J_1}^2(\tau,z) & =-c_1\frac{\dd \tau\dd
  \bar{\tau}}{(\tau-\bar{\tau})^2}+\frac{2\ii
  c_2}{\tau-\bar{\tau}}(\dd
z-p\dd\tau)\times cc, ~~p= \frac{z-\bar{z}}{\tau-\bar{\tau}},\label{METRS1}\\
\label{METRS2}\dd s^2_{\mc{X}^J_1}(x,y,p,q) & \!=\!
c_1\frac{\dd x^2\!+\!\dd   y^2}{4y^2} +\frac{c_2}{y}\left[(x^2+y^2)\dd p^2+\dd q^2+2x\dd
                                              p\dd q\right]\\
 & =c_1\frac{\dd x^2+\dd   y^2}{4y^2} \!+\!c_2\frac{x^2\!+\!y^2}{y}\left[ (\dd p\!+\!\frac{x}{x^2\!+\!y^2}\dd
                                q)^2\!+\!(\frac{y\dd q}{x^2\!+\!y^2})^2   \right],\nonumber\\
\dd s_{\mc{X}^J_1}^2(x,y,\xi,\rho) &= c_1\frac{\dd x^2+\dd
  y^2}{4y^2} + \label{METRS3}\\
 &+\frac{c_2}{y}\left[\dd \xi^2+\dd \rho^2
                                     +(\frac{\rho}{y})^2(\dd x^2+\dd
                                     y^2)-2\frac{\rho}{y}(\dd x\dd
                                     \xi+\dd y\dd
   \rho)\right],\nonumber\\
  \dd s_{\mc{X}^J_1}^2(x,y,\chi,\psi) &= c_1\frac{\dd x^2+\dd
  y^2}{4y^2}  
 +\frac{c_2}{y}\left[(x\dd \psi+\dd \chi)^2+y^2\dd \psi^2 \right].\label{METRS4}
\end{align}
\end{subequations}
\end{Proposition}
\begin{proof}
 In the
  expression \eqref{BFR2} of $B(v,u)$, $v\in\mc{X}_1$, $u\in\C$,  we
  introduce the  parametrizations of $v,u$ appearing in
  Proposition \ref{PRFC} and then we pass from the \Ka~two-form
  \eqref{BFR}
  of \Ka--Berndt 
  to the associated Riemannian metric with the standard formula \eqref{asm}.
  
   In Proposition \ref{PRFC} a) we have obtained from the
  \Ka~two-form  $\omega_{\mc{D}^J_1}$  \eqref{kk1}  the K\"ahler two-form $\omega_{\mc{X}^J_1}$
    \eqref{BFR}, from which we get the metric
      \eqref{METRS1} with formula \eqref{asm}.

       In formula \eqref{BUVpq} we introduce \eqref{UVPQ} and we get
      \eqref{BUVpq1}. With \eqref{asm} we find \eqref{METRS2}.

       In formula \eqref{BUVpq} we introduce \eqref{uxieta} and
      with 
      \eqref{LASt1},  we get \eqref{LAST12}. With
      \eqref{asm},   we find \eqref{METRS3}.

       In \eqref{METRS3} we introduce \eqref{LASt}, or in
      \eqref{METRS2} we introduce \eqref{CNI1}, and we get
      \eqref{METRS4}.

      Also, if in \eqref{BUVpq1} we take into consideration
      \eqref{CNI1}, we get \eqref{CNI11}. We apply \eqref{asm} to get \eqref{METRS4}.
  \end{proof}

Below we reproduce  the Comment 5.5 in \cite{SB19b} with some
completions: 
\begin{Comment}\label{CM1}
Berndt  \cite[p 8]{bern84}  considered the closed two-form
$\Omega=\dd \bar{\dd} f'$
of Siegel--Jacobi upper half-plane,  $G^J(\R)_0$-invariant to the action \eqref{AC1},
 obtained from the K\"ahler potential 
\begin{equation}\label{POT}
f'(\tau,z)= c_1\log (\tau-\bar{\tau}) -\ii
c_2\frac{(z-\bar{z})^2}{\tau-\bar{\tau}}, ~c_1,c_2>0.\end{equation}
Formula \eqref{POT} is  presented by Berndt as 
``communicated to the author  by \Ka''. Also in \cite[p 8]{bern84} is 
given  our  equation \eqref{METRS1}, while our \eqref{METRS2} corrects two
printing errors in Berndt's paper.

Later, in \cite[\S~ 36]{cal3}, reproduced also
in \cite{cal},
\Ka~    argues how to choose the  potential as in
\eqref{POT},   
see  also  \cite[(9)  \S ~ 37]{cal3}, 
where  $c_1=-\frac{k}{2}$, $c_2=\ii\nu\pi$, i. e.
\begin{equation}\label{POT1}
  f'(\tau,z)= -\frac{k}{2}\log\frac{ \tau-\bar{\tau}}{2\ii}
  -\ii\pi\nu \frac{(z-\bar{z})^2}{\tau-\bar{\tau}}.\end{equation}

Once the \Ka~ potential \eqref{POT1} is known, we apply the recipe
\eqref{KALP2}
$$-\ii
\omega_{\mc{X}^J_1}(\tau,z)=f'_{\tau\bar{\tau}}\dd
\tau\wedge\dd \bar{\tau}+f'_{\tau\bar{z}}\dd \tau\wedge\dd \bar{z}
-\bar{f}'_{\tau\bar{z}}\dd \bar{\tau}\wedge \dd z+f'_{z\bar{z}}\dd z\wedge
\dd \bar{z}.$$

 The metric {\emph{(8)}} in \cite{cal3}  differs from  the metric
 \eqref{METRS} by a factor of two,
 since  the Hermitian metric used by  \Ka~  is
 $\dd s^2=2g_{i\bar{j}}\dd z_i\otimes \dd\bar{z}_j$. If in \eqref{POT1} we
 take $k/2\rightarrow k$, we have 
\begin{subequations}
  \begin{align*}
 f'_{\tau}&=-k\frac{1}{\tau-\bar{\tau}}+\ii \pi \nu
    \frac{(z-\bar{z})^2}{(\tau-\bar{\tau})^2},~
    f'_{\tau\bar{\tau}}=-k\frac{1}{(\tau-\bar{\tau})^2}+2\ii
            \pi\nu\frac{(z-\bar{z})^2}{(\tau-\bar{\tau})^3}, \\
 f'_{\tau\bar{z}}& =-2\ii
    \pi\nu\frac{z-\bar{z}}{(\tau-\bar{\tau})^2},~
    f'_{z}  =-2\ii\pi\nu\frac{z-\bar{z}}{\tau-\bar{\tau}}, ~f'_{z\bar{z}}=2\ii\pi\nu\frac{1}{\tau-\bar{\tau}},
 \end{align*}
\end{subequations}
and  we get  \eqref{BFR}. 
 Relation \eqref{BFR}
 has been  obtained by Berndt  \cite[p 30]{bern}, where the
 denominator of the first term is misprinted  as $v-\bar{v}$ 
(or $\tau-\bar{\tau}$ in our notations). Equation
 \eqref{METRS2} appears also in  \cite[p 30]{bern} and \cite[p 62]{bs}.

 We also recall that in  \cite[(9.20)]{jac1} we observed that the
 \Ka~potential \eqref{POT1} should correspond to a reproducing kernel
 \begin{equation}\label{Ktau}K(\tau,z)=
   y^{-\frac{k}{2}}\exp(2\pi p^2y).\end{equation}
 In \cite[(4.3)]{gem}, see also \cite[Proposition
 4.1]{gem},  we have presented a
 generalization of \eqref{Ktau} for $\mc{X}^J_n$, obtained by Takase in \cite[\S 9]{tak}.

 Yang
 calculated  in \cite{Y07}  the 
  metric on $\mc{X}^J_n$, invariant to the action of $G^J_n(\R)_0$.  The equivalence of the metric of Yang with the metric
  obtained via CS  on $\mc{D}^J_n$ and then transported  to
  $\mc{X}^J_n$ via partial Cayley transform is underlined in
  \cite{nou}. In particular, the metric 
  \eqref{METRS3}
 appears in \cite[p 99]{Y07} for the particular values $c_1=1$,
 $c_2=4$. See also \cite{Yan,Y08,Y10}. 
\end{Comment}

\vspace{2ex}
We recall that a  vector field $X$ on a Riemannian manifold $(M,g)$  is
 called an {\it infinitesimal isometry} or a  {\it   Killing
vector field} if   the local 1-parameter group of local
transformations by $X$ in a neighbourhood   of each point of $M$ consists
of local isometries \cite[Proposition 3.2]{kn1}, i.e. 
 \begin{equation}\label{LX}L_Xg=0,~~X\in\got{D}^1(M),\end{equation} where  $L_X$ is the Lie
 derivative on $M$. 
 The condition \eqref{LX} for a vector field~\eqref{SUMXX}
\begin{equation}\label{SUMXX}X=\sum_{i=1}^nX^i\frac{\pa }{\pa
    x^i} \end{equation} 
 to be a~Killing vector field amounts to  its {\it contravariant}  
 components to verify the equations \cite[ A.6]{SB19b}
 \begin{gather} \label{LG1}
   X^{\mu}\pa_{\mu}g_{\lambda\chi}+g_{\mu\chi}\pa_{\lambda}X^{\mu}+g_{\lambda\mu}\pa_{\chi}X^{\mu}=0,
   \qquad \lambda,\chi,\mu =1,\dots, \dim{M}=n.
 \end{gather}

It is well known that {\it
   the fundamental vector field~$X^*$ on a~Riemannian homogeneous
   manifold is a~Killing vector},  see e. g.  \cite[p~4]{btv},
 \cite[Proposition~2.2, p~139]{koda} or  \cite[Remark A.4]{SB19b}. 

The next proposition completes  \cite[Proposition 5.6]{SB19b}.
\begin{Proposition}\label{Pr5}
The   three-parameter metric on  the extended  Siegel--Jacobi upper
  half -plane  
  $\tilde{\mc{X}}^J_1$, in the  S-coordinates $(x,y,p,q,\kappa)$
\begin{subequations}\label{linvG}
\begin{align}\dd s^2_{\tilde{\mc{X}}^J_1}& =\dd
    s^2_{\mc{X}^J_1}(x,y,p,q)+\lambda^2_6(p,q,\kappa)\\
  &=\frac{\alpha}{y^2}(\dd x^2+\dd
    y^2)+[\frac{\gamma}{y}(x^2+y^2)+\delta q^2]\dd p^2+
    (\frac{\gamma}{y}+\delta p^2)\dd q^2 +\delta \dd \kappa^2\\
& + 2(\gamma\frac{x}{y}-\delta pq)\dd p\dd q +2\delta (q\dd p\dd
  \kappa-p\dd q \dd \kappa) \nonumber\end{align}\end{subequations}
 is left-invariant with  respect  to  the action of the Jacobi group
 $G^J_1(\R)$  given in {\emph{Lemma\,\ref{LEMN}}}.

If $$X=X^1\pa_x+X^2\pa_y+X^3\pa_p+X^4\pa_q+X^5\pa_{\kappa}\in\got{g}^J_1(\R)\ominus
<L_3>,$$
 then the fundamental vector fields  \eqref{EQQ31xx}  verify the  following Killing equations \eqref{LG1}   with
 respect to the metric \eqref{linvG}, invariant to the action
 \eqref{AC2}:
 %%%%%%
%%%\begin{subequations}\label{DFFT}
  \begin{align*}
    & -X^2+y\pa_xX^1=0,\nonumber\\% 11
    & \pa_xX^2+\pa_yX^1=0, \nonumber\\ %12
    & (\gamma\frac{x^2+y^2}{y}+\delta
    q^2)\pa_xX^3+\delta_q\pa_xX^5+(\gamma\frac{x}{y}-\delta
      pq)\delta_xX^4=0, \nonumber\\ % 13
    & (\gamma \frac{x}{y}-\delta p q)\pa_x
      X^3+(\frac{\gamma}{y}+\delta p^2)\pa_x
      X^4+\frac{\alpha}{y^2}\pa_qX^2=0, \nonumber\\ % 14
    & \delta pq\pa_xX^3-\delta p \pa _x X^4+\delta \pa_x X^5
    +\frac{\alpha}{y^2}\pa_{\kappa}X^1=0, \nonumber\\ %15
    & -X^2+y \pa_yX^2=0, \nonumber\\ %%22
     &(\gamma\frac{x}{y}-\delta pq)\pa_y
      X^4+(\gamma\frac{x^2+y^2}{y}+\delta q^2)\pa_y X^3+\delta_q
       \pa_{\kappa}X^5+\frac{\alpha}{y^2}\pa_pX^2=0, \nonumber\\ % 23
    &(\gamma\frac{x}{y}-\delta pq)\pa_yX^3+(\frac{\gamma}{y}+\delta
    p^2)\pa_yX^4-\delta p \pa_y
      X^5+\frac{\alpha}{y^2}\pa_qX^2=0,\nonumber\\ %%%24
    &  \delta q
      \pa_yX^3-\delta_p\pa_yX^4+\delta\pa_yX^5+\frac{\alpha}{y^2}\pa_kX^1=0,
      \nonumber\\ %%%25 ok
    & 2\gamma \frac{x}{y}X^1+\gamma(1-\frac{x^2}{y^2})X^2+2\delta q
      X^4+2(\gamma\frac{x^2+y^2}{y}+\delta
      q^2)\pa_pX^3+2(\gamma\frac{x}{y}-\delta pq)\pa_pX^4+\nonumber \\ &2\delta q
      \pa_p X^5=0,
      \nonumber\\   %%% 33 ok
    & \gamma\frac{x}{y}(yX^1-xX^2)-\delta(qX^3+pX^4)+2(\gamma\frac{x}{y}-\delta
    pq)\pa_pX^3+(\frac{\gamma}{y}+\delta p^2)\pa_pX^4-\delta p\delta_p
      \pa_pX^5 +\nonumber\\
    & \gamma(\frac{x^2+y^2}{y}+\delta
    q^2)\pa_qX^3+(\gamma\frac{x}{y}-\delta pq)\pa_qX^4
      +\delta_q\pa_qX^5 =0,\nonumber\\%%%34
 & \delta X^4+\delta q \pa_p X^3-\delta p\pa_p X^4 +\delta \pa_p X^5
    +(\gamma\frac{x^2+y^2}{y}+\delta q^2)\pa_{\kappa}X^3 +
    (\gamma\frac{x}{y}-\delta pq)\pa_{\kappa}X^4+\nonumber \\ & \delta q\pa_kX^5=0, \nonumber\\ %%%35
    &  -\frac{\gamma}{y^2}X^2+2 \delta p X^3+2(\gamma\frac{x}{y}-\delta
      pq)\delta_qX^3+2(\frac{\gamma}{y}+\delta p^2)\pa_qX^4-2\delta pq
      \pa_qX^5=0, \nonumber\\ %%%44
    & \delta X^3 +\delta q \pa _qX^3-\delta p \pa_qX^4+\delta
      \pa_qX^5+(\gamma\frac{x}{y}-\delta
      pq)\pa_{\kappa}X^3+(\frac{\gamma}{y}+\delta
      p^2)\pa_{\kappa}X^4-\delta_p\pa_{\kappa}X^5=0,\nonumber\\  %%%45
    &  q \pa_{\kappa}X^3-p\pa_kX^4+\pa_kX^5=0.\nonumber  %%%%55
     \end{align*}
  %%%\end{subequations}

\end{Proposition}

The next proposition is a completion of \cite[Theorem 5.7]{SB19b}.
\begin{Proposition}\label{BIGTH}

The four-parameter  left-invariant metric on the real  Jacobi group  $G^J_1(\R)$ in the 
$S$-coordinates $(x,y,\theta,p,q,\kappa)$  is
\begin{equation}\label{MTRTOT}
\begin{split}
\dd
s^2_{G^J_1(\R)} &  =\sum_{i=1}^6\lambda_i^2\\
 & =\alpha\frac{\dd
x^2+\dd y^2}{y^2}   +\beta(\frac{\dd x}{y}+2\dd
\theta)^2 \\
+& \frac{\gamma}{y}[\dd q^2+(x^2+y^2)\dd p^2+2x\dd p\dd
q]+\delta(\dd \kappa-p\dd q+q\dd p)^2.
\end{split}
\end{equation}

We have $<\lambda_i|L^j>=\delta_{ij}, i,j=1,\dots ,6$, where  the vector fields $L^i$, $i=1,\dots,6$
verify the commutation relations \eqref{FUFUFU} and 
are orthonormal with respect to the  metric \eqref{MTRTOT}.

If
$$X=X^1\pa_x+X^2\pa_y+X^3\pa_{\theta}+X^4\pa_p+X^5\pa_q
 +X^6\pa_{\kappa}\in \got{g}^J_1(\R),$$
then the fundamental  vector fields \eqref{NEWPQR4}   in the $S$-variables, associated to  the
generators $F,G,H,P,Q,R$,  corresponding to the left
action \eqref{COMPL}  of the Jacobi group on himself  are
solution of the following Killing equations  associated to the invariant metric
\eqref{MTRTOT}
 %%%\begin{subequations}\label{DFFTT}
  \begin{align*}
  & -(\alpha +\beta )X^2+(\alpha+\beta )y\pa_xX^1+2\beta y^2\pa_xX^3=0,\nonumber\\% 11
    &\alpha\pa_xX^2+(\alpha+\beta)\pa_yX^1+2\beta y \pa_yX^3=0,
      \nonumber\\% 12
    &-2\beta X^2+2\beta y \pa_x X^1+4\beta y^2\pa_xX^3+(\alpha+\beta)\pa_{\theta}X^1+2\beta
     y\pa_{\theta}X^3=0, \nonumber\\% 13
     &  (\gamma\frac{x^2+y^2}{y}+\delta
    q^2)\pa_xX^4+(\gamma\frac{x}{y}-\delta p q)\pa_x X^5+\delta
    q\pa_{x}X^6+\frac{\alpha+\beta}{y^2}\pa_pX^1+2\frac{\beta}{y}\pa_pX^3=0,
       \nonumber\\% 14
    &(\gamma\frac{x}{y}-\delta pq)\pa_xX^4+(\frac{\gamma}{y}+\delta
      p^2)\pa_xX^6-\delta p \pa_xX^6+\frac{\alpha+\beta}{y^2}\pa_qX^1+2\frac{\beta}{y}\pa_q
      X^3=0,  \nonumber \\ % 15
    & \delta q \pa_xX^4 -\delta p\pa_xX^5+\delta\pa_xX^6+\frac{\alpha
      +\beta}{y^2}\pa_{\kappa}X^3+2\frac{\beta}{y}\pa_{\kappa}X^3,
      \nonumber \\ % 16
    &-X^2+ y\pa_yX^2=0, \nonumber \\ % 22
    &2\frac{\beta}{y}\pa_yX^1+4\beta\pa_yX^3+\frac{\alpha}{y^2}\pa_{\theta}X^2
      =0,  \nonumber \\ % 23
     &(\gamma\frac{x^2+y^2}{y}+\delta
      q^2)\pa_yX^4+(\gamma\frac{x}{y}-\delta pq)\pa_yX^5+\delta q
      \pa_yX^6+\frac{\alpha}{y^2}\pa_pX^2 =0, \nonumber \\ % 24
    & (\gamma\frac{x}{y}-\delta pq)\pa_yX^4+(\frac{\gamma}{y}+\delta
      p^2)\pa_yX^5-\delta p \pa_yX^6+\frac{\alpha}{y^2}\pa_qX^2=0,
      \nonumber \\ % 25
    & \delta q \pa_yX^4-\delta p \pa_y X^5+\delta \pa_y
      X^6+\frac{\alpha}{y^2}\pa_{\kappa}X^2=0,  \nonumber \\ % 26
    & \pa_{\theta}X^1 +2y\pa_{\theta}X^3=0, \nonumber \\ % 33
    & (\gamma\frac{x^2+y^2}{y}+\delta
    q^2)\pa_{\theta}X^4+(\gamma\frac{x}{y}-\delta pq)\pa_{\theta}X^5
    +\delta
    q\pa_{\theta}X^6+2\frac{\beta}{y}\pa_pX^1+4\beta\pa_pX^6=0,
      \nonumber \\ % 34
    &  (\gamma\frac{x}{y}-\delta
      pq)\pa_{\theta}X^4+(\frac{\gamma}{y}+\delta
      p^2)\pa_{\theta}X^5-\delta
      p\pa_{\theta}X^6+2\frac{\beta}{y}\pa_qX^3+4\beta\pa_qX^3=0,
      \nonumber \\ % 35
     & \delta q\pa_{\theta}X^4+(\frac{\gamma}{y}-\delta
    p)\pa_{\theta}X^5+\delta \pa_{\theta}X^6+4\beta
      \pa_{\kappa}X^3+2\frac{\beta}{y}\pa_{\kappa}X^1=0, \nonumber \\
      % 36
    & 2\gamma\frac{x}{y}X^1+\gamma(1-\frac{x^2}{y^2})X^2+2\delta q
        X^5+2(\gamma\frac{x^2+y^2}{y}+\delta
      q^2)\pa_pX^4\\
    & +2(\gamma\frac{x}{y}-\delta pq)\pa_pX^4+2\delta q
      \pa_pX^6=0, \nonumber \\ % 44
    &\frac{\gamma}{y}X^1-\gamma\frac{x}{y^2}X^2-\delta(qX^4+pX^5)+(\gamma\frac{x}{y}-\delta
      pq)\pa_pX^4+(\frac{\gamma}{y}+\delta p^2)\pa_pX^5 -\delta
      p\pa_pX^6 +\nonumber \\
   ~~~~~ & (\gamma\frac{x^2+y^2}{y}+\delta q^2)\pa_q X^4 +
      (\gamma\frac{x}{y}-\delta pq)\pa_qX^5+\delta
           q\pa_qX^6=0,\nonumber \\ %%%45
    & \delta q \pa_p X^4-\delta p \pa_p X^5 +\delta \pa_p
      X^6+(\gamma\frac{x^2+y^2}{y}+\delta
      q^2)\pa_{\kappa}X^4+(\gamma\frac{x}{y}-\delta
      pq)\pa_{\kappa}X^5+\delta q \pa_{\kappa} X^6=0, \nonumber \\ %% 46
    & -\frac{\gamma}{y^2}X^2+2\delta pX^4+2(\gamma\frac{x}{y}-\delta
      pq)\pa_{q}X^4+2(\frac{\gamma}{y}+\delta p^2)\pa_qX^4-2\delta p
      \pa_q X^6=0, \nonumber  \\ %%% 55
    & -\delta X^4+\delta q \pa_q X^4-\delta p
      \pa_qX^5+\delta\pa_qX^6+(\gamma\frac{x}{y}-\delta
      p\kappa)\pa_{\kappa}X^4\\
    & +(\frac{\gamma}{y}+\delta p^2)\pa_{\kappa}X^5-\delta p\pa_{\kappa}X^6=0,\nonumber \\ %%% 56
    & q\pa_{\kappa}X^6-p\pa_{\kappa}X^5+\pa_{\kappa}X^6=0.\nonumber %%% 66
    \end{align*}
  %%%\end{subequations}
 \end{Proposition}

\section{Natural reductivity and geodesic vectors  on  $\tilde{\mc{X}}^J_1$}\label{NRGO}
We briefly recall the notions of natural reductivity and geodesic
vectors. More references and details  are given in \cite[Appendix A]{SB19b}.

In accord with  Nomizu \cite{nomizu}, 
a homogeneous space $M=G/H$ is {\em reductive} if the Lie algebra \g~
of $G$ may be decomposed into a vector space direct sum of the Lie
algebra \h ~ of $H$ and an $\Ad (H)$-invariant subspace \m , that is
\begin{subequations} 
\begin{equation}\label{sum1}
\g = \h \oplus \m, ~ \h\cap\m =\emptyset,
\end{equation}
\begin{equation}\label{sum2}
 \Ad(H)\m \subset \m. 
\end{equation}
{\text{Condition (\ref{sum2}) implies}}
\begin{equation}\label{sum3}
[\h ,\m ]\subset \m, 
\end{equation}
\end{subequations}
and  conversely, if $H$ is connected, then (\ref{sum3}) implies
 (\ref{sum2}). 

If the Lie algebra $\got{g}$ and its subalgebra $\got{h}$ associated with
the homogenous   manifold $M=G/H$ satisfy \eqref{sum1}, then a
necessary and sufficient condition for $M$ to be a  {\it locally symmetric
  space} is
\begin{equation}\label{CUP} [\got{h}, \got{h}]\subset \h,\quad [\h,\m]\subset \m,\quad
  [\m,\m]\subset \h.
\end{equation}

{\it If $M$ is a complete, simply connected Riemannian locally symmetric
space, then $M$ is a Riemannian globally symmetric space}
\cite[Theorem 5.6 p 232]{helg}.

We recall   \cite[Theorem 5.4]{as},  \cite[Appendix A.4]{SB19b}, \cite[Proposition 1 p 5]{btv},
\cite[Ch X, \S 3]{kn}, \cite[Theorem 6.2, 
p 58]{tv}.

{\it Let $(M,g)$ be a homogeneous Riemannian manifold.  Then $(M,g)$ is
a naturally reductive Riemannian homogenous space if and only if there
exists a connected Lie subgroup $G$ of $I(M)$ acting transitively and
effectively  on $M$ and
a reductive decomposition \eqref{sum1}   such that one of the
following equivalent statements holds}{\rm :}
\begin{enumerate}\itemsep=0pt
\item[$(i)$] {\it the following relation is verified}
\begin{equation}\label{X2X3}g([X_1,X_3]_{\m},X_2)+g(X_1,[X_3,X_2]_{\m})=0,~\forall
X_1,X_2,X_3\in\m; \end{equation}
 \item[$(ii)$] $(*)$  {\it every geodesic in $M$ is the orbit of a one-parameter subgroup of $I(M)$ generated by some} $X\in\m$.
\end{enumerate} 
 
The natural reductivity is a special case of spaces with a more
general property than~$(*)$, see~\cite[Appendix A.7]{SB19b}, \cite{kwv}:\\
$(**)$ {\it Each geodesic of} $(M =G/H,g)$ {\it is an orbit of a one parameter group of isometries} $\exp tZ$, $Z\in\got{g}$.

A vector $X \in\got{g}\setminus \{0\}$ is called a {\it geodesic
  vector} 
if the curve   $\gamma(t)=(\exp tX)(p)$ is a geodesic,
cf. \cite{kwv}. 
Riemannian homogeneous spaces with property (**) are called {\it
  g. o. spaces}. (g. o. =  geodesics are orbits). All naturally
reductive spaces are g. o. spaces.

Kowalski and
Vanhacke \cite{kwv} have proved that the condition to have a geodesic vector is expressed in 
the

{\it Geodesic Lemma:   On homogeneous Riemannian manifolds} $M=G/H$ {\it a vector} $X
\in\got{g}\setminus \{0\}$
{\it is geodesic if and only if}
\begin{equation}\label{BCOND}
B([X,Y]_{\m}, X_{\m})=0, \forall Y\in \m.
\end{equation}

It is known \cite{kwv} that:
{\it Every simply connected Riemannian g. o.~space $(G/H,g)$ of dimension $n\le 5$ is a naturally reductive Riemannian manifold.}

The next theorem  is a completion of \cite[Proposition 5.8]{SB19b}
\begin{Theorem}\label{PRLST}
a) The Siegel--Jacobi upper half-plane, 
realized as homogenous Riemannian manifold
$(\mc{X}^J_1=\frac{G^J_1(\R)}{\text{SO}(2)\times\R}, g_{\mc{X}^J_1})$, 
is a reductive, non-symmetric manifold, not  naturally reductive with
respect to the balanced metric \eqref{METRS2}.
%%%%%%%%% 10 Jan 209

The Siegel--Jacobi upper half-plane  $\mc{X}^J_1$ is not a 
g. o. space  with respect to the balanced metric.

b) If
 \begin{equation}\label{XLP}
\got{g}^J_1(\R)\ni X=aL^1+bL^2+cL^3+dL^4+eL^5+fL^6,
\end{equation}
then the  geodesic vectors of the homogeneous manifold $\mc{X}^J_1$ have
one of the following expressions given in {\emph{Table 1}}
%%%\begin{subequations}
%%%\begin{align}
%%%X & = a(L^1-L^3+\sqrt{\frac{\alpha}{\beta}}L^5)+fL^6,\\
%% & = a(L^1+L^3+ \sqrt{\frac{\alpha}{\beta}}L^4)+fL^6,\\
 %%%& = aL^1+bL^2+fL^6.
%%%\end{align}
%%%5\end{subequations}

\begin{center}
\begin{tabular}{||c|c|c|c|c|c|c||}
 \hline
\multicolumn{7}{|c|}{\small {\emph{Table 1: Components of the geodesic  vectors}}
  \eqref{XLP} {\emph{on}} $\mc{X}^J_1$}\\
  \hline Nr. cr. & {a} & {b} &
{c} & {d} & {e} & {f} \\
\hline {\rm 1} & $0$ & $0$ & $c$  &  $0$ & $0$ & $f$ \\
\hline {\rm 2}& $a$ & $b$ & $0$ &  $0$ &  $0$ & $f$ \\
\hline {\rm 3}& $rc$ & $0$  & $c$  & $\pm rc$& $0$ & $f$  \\
\hline {\rm 4}& $a$ & $0$  & $-a$  & $0$ &$\epsilon\sqrt{r}a$ & $f$ \\
\hline {\rm 5}& $\epsilon_1\epsilon_2 \frac{1-r}{\sqrt{r}}e$ & $\epsilon_1 e$
                                 & $-\frac{\epsilon_1\epsilon_2}{\sqrt{r}}e$
    & $\epsilon_2\sqrt{r}e$ & $e$ & $f$\\
\hline
\end{tabular}
\end{center}
where $r=\sqrt{\frac{\alpha}{\beta}}$,
$\epsilon_1^2=\epsilon_2^2=\epsilon^2=1$.

c) The extended Siegel--Jacobi upper half-plane, realized as homogenous 
Riemannian  manifold
$(\tilde{\mathcal{X}}^J_1=\frac{G^J_1(\mathbb{R})}{\text{SO}(2)},g_{\tilde{\mc{X}}^J_1})$,
is a
reductive, non-symmetric manifold, non-naturally reductive with
respect to the metric \eqref{linvG}.

The Siegel Jacobi upper half-plane  $\tilde{\mc{X}}^J_1$ is not a 
g. o. space  with respect to the  invariant  metric \eqref{linvG}.

d) If we take $X\in\got{g}^J_1$ as in \eqref{XLP}, then the geodesic
vectors on the extended Siegel--Jacobi manifold  $\tilde{\mc{X}}^J_1$ are given
in {\emph{Table 2}} for
\[        %%%%%%%\begin{equation}\label{R3}
r >R_3=
  [\frac{1}{2}+\frac{1}{6}(\frac{31}{3})^{\frac{1}{2}}]^{\frac{1}{3}}+[\frac{1}{2}-\frac{1}{6}(\frac{31}{3})^{\frac{1}{2}}]^{\frac{1}{3}}
     \approx 0.6823\dots.\]%%%%%%%%%55\end{equation}
    
\begin{center}
\begin{tabular}{||c|c|c|c|c|c|c||}
 \hline
\multicolumn{7}{|c|}{\small {\emph{Table 2: Components of the geodesic
  vectors}}
  \eqref{XLP}  {\emph{on}} $\tilde{\mc{X}}^J_1$}\\
\hline Nr. cr. & {a} & {b} & {c} & {d} & {e} & {f} \\
\hline {\rm 1}&
          $\epsilon_1\epsilon_2(1-r)\sqrt{\frac{1+r^2}{rF_2}}e$&
                                                                        $\epsilon_2\sqrt{\frac{F_3}{r(r^2+1)}}e$ & $-\epsilon_1\epsilon_2r\sqrt{\frac{r}{(r^2+1)F_2}}e$ & $\epsilon_1\sqrt{\frac{F_3}{F_2}}e$  &  $e\not= 0$ &   $f$ \\
\hline {\rm 2} & $\epsilon\sqrt{\frac{r}{r^2+2}}e$ & $0 $ & $ -\epsilon\sqrt{\frac{r}{r^2+2}}e$ &  $d= 0$ &  $e\not=0$ & $f$ \\
\hline {\rm 3} &  $rc$ & $0$  & $c\not= 0$ & $\epsilon\sqrt{2+r^2}c$ & $0$ & $f$  \\
  \hline {\rm 4} &  $0$ & $0$  & $c$ & $0$ & $0$ & $f$  \\
  \hline {\rm 5} &  $0 $ & $b$  & $0$ & $0$ & $0$ & $f$  \\
  \hline {\rm 6} &  $a $ & $0$  & $0$ & $0$ & $0$ & $f$  \\
  \hline
\end{tabular}
\end{center}
The polynomials    $F_2,F_3$ 
are defined in \eqref{star2}
\begin{equation}\label{star2}
  F_3(r)=r^3+r-1,~F_2(r)=r^2-r+1,~r\in\R.
  \end{equation}
\end{Theorem}
\begin{proof}
  a) This part has been proved in \cite{SB19b}.

  b) This part was  also  proved in \cite{SB19b}. Here we give
  more details which are used in the next parts of the theorem.

To find the geodesic vectors on the Siegel--Jacobi upper half-plane
$\mc{X}^J_1$, we look for the solution \eqref{XLP}  that verifies
the condition  \eqref{BCOND}.

Taking 
\[
\m\ni Y= a_1L^1+b_1L^2+d_1L^4+e_1L^5,
\]
the condition \eqref{BCOND}  becomes
%%%%%%% \begin{multline*}
\[
  a_1(\frac{bc}{\sqrt{\beta}}\!+\!\frac{ed}{\sqrt{\alpha}})\!+\!\frac{b_1}{2}[-\frac{ac}{\sqrt{\beta}}\!+\!\frac{1}{\sqrt{\alpha}}(d^2\!-\!e^2)]
\!-\!\frac{d_1}{2\sqrt{\alpha}}(bd\!+\!ec\!+\!ae)
\!+\!\frac{e_1}{2}[\frac{cd}{\sqrt{\beta}}\!+\!\frac{1}{\sqrt{\alpha}}(be\!-\!ad)]\!=\!0,
\]
%%%%%%%\end{multline*}
and must be satisfied for every values of $a_1,b_1,d_1,e_1$, i.e.  the
coefficients of the geodesic vector \eqref{XLP} are solutions to the system of
algebraic equations
\begin{subequations}\label{KLJH}
\begin{eqnarray}
& &rbc+de=0,\label{KL1}\\
& &-rac+ d^2-e^2=0,\label{KL2}\\
& &bd+e(a+c)=0,\label{KL3}\\
& &rcd+be-ad=0.\label{KL4}
\end{eqnarray}
%% \right. .
\end{subequations}
%%%\begin{equation}\label{KLJH}
%%\left\{
%%%\begin{array}{l} rbc+de=0,\label{KL1}\\
 %%-rac+ d^2-e^2=0,\label{KL2}\\
%%%%%% bd+e(a+c)=0,\label{KL3}\\
%%%rcd+be-ad=0.\label{KL3}
%%\end{array}
%%%%%\right. .
%% \end{equation}
Now suppose that we have $de\not=0$. We write \eqref{KL3} and
\eqref{KL4}
as
\begin{subequations}\label{24KL}
  \begin{eqnarray}
    & & a + c=-\frac{bd}{e},\\
    & & -a+ rc= -\frac{be}{d}.
    \end{eqnarray}
  \end{subequations}
 We find  from \eqref{24KL}
  \begin{equation}\label{ABab}
    a=\frac{b}{de}\frac{e^2-rd^2}{1+r},\quad
  c=-\frac{b}{ed}\frac{d^2+e^2}{1+r}. \end{equation}
  Introducing $a$ and $c$ from \eqref{ABab} into \eqref{KL1} and
  \eqref{KL2}, we find the compatibility condition
  \begin{equation}\label{compeb}
    d^2=re^2.
  \end{equation}
  If $ed\not= 0$, we get the first  line in Table 1. The other
  situations are contained in the next lines  of  Table 1.

c) We consider 
\[       %%%%%%%%%%\begin{equation}\label{MHM1}
  \got{m}=<F,G,H,P,Q>,~\got{h}= <R>,
\]      %%%%%%%\end{equation}
and  with  the commutation relations  \eqref{PQT1}, \eqref{FGHCOM}, 
\eqref{MORCOM}, we get  $[\got{h},\got{m}]\subset \got{m}$, but
$[\got{m},\got{m}] \nsubseteq \got{h}$.
This contradicts relation \eqref{CUP} satisfied by a symmetric manifold.

In order to verify the natural reductivity of the extended Siegel
upper half-plane,  we have to check relation \eqref{X2X3}. We take
\begin{equation}\label{X11X31}
  X_i=a_iL^1+b_iL^2+c_iL^3+d_iL^4+e_iL^5\subset\m, ~i=1,2,3.
\end{equation}
%%%%%%With the commutation relations \eqref{FUFUFU}, the condition
%%%%\eqref{X2X3} reads
%%%5\begin{equation}\label{NEWSU}
 %%%%% \sum_{j=1}^5\alpha_{ij}X_j=0, ~i=1,\dots,5,\end{equation}
%%%%%where   the elements of the matrix $\{X_j\}_{j=1,\dots,6}$ are equal
%%%%55with respectively  the succesive letters $a_1,\dots,e_1$.

Let us introduce the following  
notation $$\gamma:=\frac{1}{\sqrt{\beta}}-\frac{\sqrt{\beta}}{\alpha},\quad
\zeta:=\frac{\sqrt{\beta}}{\alpha}-\frac{1}{2\sqrt{\beta}}.$$
Then \eqref{X2X3} reads
\begin{equation}\label{PASTEM}
  \left\{
  \begin{array}{l}
    \gamma(c_1b_2-b_1c_2)=0,\\
     \zeta(c_1a_2-a_1c_2)=0, \\
    \frac{3}{\sqrt{\beta}}(b_1a_2-a_1b_2)+(\frac{1}{\sqrt{\alpha}}+\frac{1}{\sqrt{\beta}})(d_1e_2-e_1d_2)=0,\\
  -a_1e_2  -b_1d_2 -c_1e_2 +d_1b_2+e_1(a_2+c_2)=0,\\
      \frac{d_1a_2-a_1d_2+b_1e_2-e_1b_2}{\sqrt{\alpha}}+\frac{c_1d_2-d_1c_2}{\sqrt{\beta}}=0
   \end{array}
   \right. .
 \end{equation}
 We write the system of algebraic equations \eqref{PASTEM} as
 \begin{equation} \sum _{j=1}^5A_{ij}x_j=0,\quad i=1,\dots,5,\end{equation}
 where $x:=(x_1,\dots,x_5)=(a_1,\dots,e_1)$.

 Now we calculate matrix $A$ from \eqref{PASTEM}
\begin{equation}\label{big5}
  A:=\{A_{ij}\}_{i,j=1,\dots,5}=
  \left(\begin{array}{ccccc}
  0 & \gamma c_2&
            - \gamma b_2& 0 &0\\
          -\zeta c_2& 0 & \zeta a_2& 0 & 0 \\
          -\frac{3b_2}{\sqrt{\beta}}&\frac{3a_2}{\sqrt{\beta}}
 & 0& (\frac{1}{\sqrt{\alpha}}+\frac{1}{\sqrt{\beta}})e_2
 & -(\frac{1}{\sqrt{\alpha}}+\frac{1}{\sqrt{\beta}})d_2\\
          -e_2& -d_2& -e_2&b_2& a_2+c_2\\
          -\frac{d_2}{\sqrt{\alpha}}& \frac{e_2}{\sqrt{\alpha}} & \frac{d_2}{\sqrt{\beta}}& \frac{a_2}{\sqrt{\alpha}}-\frac{c_2}{\sqrt{\beta}}& -\frac{b_2}{\sqrt{\alpha}}
          \end{array}\right)
      \end{equation}
 and, computing its determinant, we find  $\det{A}=0$ for any $X_2\in\got{m}$. This means that we can find $X_1\in\got{m}$ such that \eqref{X2X3} be satisfied for any $X_2,X_3\in\got{m}$, and thus $\tilde{\mc{X}}^J_1$ is not a naturally reductive space with respect to the metric  \eqref{linvG}.

  d) The condition for   a vector $X$ 
 \eqref{XLP} to be a geodesic  vector  on the
  extended Siegel--Jacobi upper half-plane   is to verify \eqref{BCOND}. If we
   take $$\m\ni Y=a_1L^1+b_1L^2+c_1L^3+d_1L^4+e_1L^5,$$ with the
   commutation relations \eqref{FUFUFU}, we find
     \begin{multline*}
      [X,Y]= -\frac{1}{2\sqrt{\beta}}(cb_1-c_1b)L^1+\frac{1}{\sqrt{\beta}}(ca_1-c_1a)L^2-\frac{\sqrt{\beta}}{\alpha}(ab_1-a_1b)L^3\\
         +[-\frac{1}{2\sqrt{\alpha}}(ae_1-a_1e)-\frac{1}{2\sqrt{\alpha}}(bd_1-db_1)+\frac{1}{2\sqrt{\beta}}(ce_1-c_1e)]L^4\\
         +[-\frac{1}{2\sqrt{\alpha}}(ad_1-da_1)+\frac{1}{2\sqrt{\alpha}}(be_1-b_1e)-\frac{1}{2\sqrt{\alpha}}(cd_1-c_1d)]L^5\\
         + 2(de_1-ed_1)\frac{\sqrt{\delta}}{\gamma}L^6 .{\qquad\qquad \qquad\qquad\qquad \qquad\qquad\qquad \qquad\qquad\qquad}
       \end{multline*}
       Condition \eqref{BCOND}
     requires   the components of the geodesic vector $X$ to verify
     the system of algebraic   equations:
     \begin{subequations}\label{KLJH1}
%%%\left\{
\begin{eqnarray}
& & (r+\frac{1}{r})bc+de=0,\label{JHH11}\\
& & -(r+\frac{2}{r})ac+ d^2-e^2=0,\label{JHH22}\\
& &    -rab+(1-r)de=0,\label{JHH33}\\
& &  bd+e(a+c)=0,\label{JHH44}\\
& & rcd+be-ad=0.\label{JHH55}
\end{eqnarray}
%%%\right. .
\end{subequations}
      From \eqref{JHH44} and \eqref{JHH55} we get for $a$ and $c$ the
     expressions given in \eqref{ABab}, which we introduce in \eqref{JHH11} and obtain
     \begin{equation}\label{mat1}
     \frac{b^2}{d^2e^2}=\frac{r(r+1)}{r^2+1}\frac{1}{d^2+e^2}.
     \end{equation} 
     We also introduce in \eqref{JHH33} the expressions for $a$
       and $c$ given in  \eqref{ABab} and we get
     \begin{equation}\label{mat2}
       \frac{b^2}{d^2e^2}=\frac{1-r^2}{r}\frac{1}{e^2-rd^2}.
     \end{equation}

     The compatibility of equations \eqref{mat1} and \eqref{mat2}
     imposes the following restriction:
     \begin{equation}\label{star1}
       \frac{d^2}{\e^2}= \frac{F_3(r)}{F_2(r)}.
      \end{equation}
     The real root $R_3$ of equation $F_3(r)=0 $ is obtained with
     Cardano's formula as
     \[ %%%%%%%\begin{equation}\label{R33}
       R_3=[\frac{1}{2}+\frac{1}{6}(\frac{31}{3})^{\frac{1}{2}}]^{\frac{1}{3}}+[\frac{1}{2}-\frac{1}{6}(\frac{31}{3})^{\frac{1}{2}}]^{\frac{1}{3}}
     \approx 0.6823\dots.\]%%%%%%%%\end{equation}
    
   Introducing \eqref{ABab} in \eqref{JHH22}, we come back  to 
   condition \eqref{star1}.
    
 \end{proof}

 In conclusion,
 \begin{Comment}\label{CM2}
   In this paper we have investigated some geometric properties of the
   extended  Siegel-Jacobi upper half-plane introduced
   in \cite{SB19b}. If the invariant metric on the four dimensional
   manifold $\mc{X}^J_1$ can be obtained with the CS methods, the invariant metric
   on the five dimensional manifold $\tilde{\mc{X}}^J_1$ can be
   obtained only  with Cartan's moving frame method.  Both manifolds  $\mc{X}^J_1$ and
   $\tilde{\mc{X}}^J_1$ are reductive, non-symmetric, non-naturally reductive
   manifolds and consequently are not g. o. spaces. 
   \end{Comment}

\subsection*{Acknowledgements}

This research  was conducted in  the  framework of the 
ANCS project  program   PN 19 06
01 01/2019.

%%%%%%%%%%%


\begin{thebibliography}{99}

 \footnotesize\itemsep=0pt

\bibitem{as} W. Ambrose, I.M.  Singer, {\it On homogeneous
    Riemannian manifolds}, Duke Math. J. {\bf 25} (1958)  647--669 



\bibitem{arr} C. Arezzo,  A.   Loi,  {\it  Moment maps, scalar curvature
    and quantization of K\"ahler manifolds},  Commun. Math. Phys. {\bf  246} (2004) 543--549


 


\bibitem{ball}W. Ballmann, 
{\it Lectures on K\"ahler Manifolds},  ESI Lectures in Mathematics and
Physics, European Mathematical Society (EMS), Z\"urich, 2006

\bibitem{ber97A}
S. Berceanu, {\it On the geometry of complex Grassmann manifold, its
  noncompact dual and coherent states}, Bull. Belg. Math. Soc {\bf 4}
(1997) 205--243

\bibitem{ber97}
S. Berceanu, {\it Coherent states and geodesics: cut locus and conjugate locus},
J.~Geom. Phys. \textbf{21} (1997)
149--168, arXiv:dg-ga/9502007


\bibitem{BERA} S. Berceanu,   Coherent states, phases and symplectic areas of geodesic
triangles, in 
{\it Coherent States, Quantization and Gravity},  
Editors  M. Schlichenmaier, A. Strasburger, S.T.
 Ali, A. Odjziewicz, Warsaw University Press, Warsaw, 2001,
 129--137,  arXiv:math.DG/9903190


\bibitem{SB03} S.  Berceanu, Realization of coherent state algebras by
  differential operators, in {\it Advances in Operator Algebras and
    Mathematical Physics}, Sinaia, 2003. Editors  F. Boca, O. Bratteli, R. Longo, H. Siedentop, The Theta Foundation, Bucharest 2005, pp 1--24; 	arXiv:math/0504053 [math.DG]



\bibitem{BERB} S. Berceanu, Geometrical  phases on Hermitian symmetric spaces,
in {\it  Recent Advances in Geometry and Topology}, 
   Editors D. Andrica,
P. A. Blaga,   Cluj University Press, Cluj-Napoca, 2004,
 83--98, arXiv:math.DG/0408233


\bibitem{SB05}  S. Berceanu, {\it A holomorphic representation of Lie algebras
semidirect sum of semisimple and Heisenberg algebras},  Romanian
J. Phys. {\bf 50} (2005) 81--94





\bibitem{jac1} S. Berceanu,  {\it A holomorphic representation of the Jacobi algebra},
Rev. Math. Phys.  {\bf 18}  (2006) 163--199; {\it  Errata},   Rev.  Math. Phys.
{\bf 24}  (2012) 1292001, 2~pages, arXiv:math.DG/0408219




\bibitem{mlad} S. Berceanu, {\it Coherent states associated to the Jacobi group -} 
{\it a variation on a theme by Erich K\"ahler},
 J. Geom.  Symmetry  Phys.  {\bf 9} (2007)
1--8 



\bibitem{BER77}  S. Berceanu, {\it  Coherent states associated to the Jacobi group},
Romanian Rep.  Phys. {\bf 59}   (2007)  1089--1101 

\bibitem{sbj}
 S. Berceanu,   A holomorphic 
representation of Jacobi algebra in several 
dimensions, in {\it  Perspectives in Operator Algebra and Mathematical
Physics}, Editors   F.-P. Boca, R. Purice,  S. Stratila, \textit{Theta
Ser. Adv. Math.}, Vol.~8, Theta, Bucharest, 2008, 1--25, arXiv:math.DG/0604381





\bibitem{nou} S. Berceanu,  {\it  A convenient coordinatization of Siegel--Jacobi
    domains},   Rev.  Math. Phys.  {\bf 24}  (2012) 1250024, 38 pages,
  arXiv:1204.5610



\bibitem{FC} S. Berceanu,  {\it Consequences of the fundamental conjecture for the motion
  on the  Siegel--Jacobi disk},   Int.  J.  Geom. Methods Mod. Phys. {\bf
  10} (2013) 1250076, 18 pages,  arXiv:1110.5469


 \bibitem{SB14}S. Berceanu, {\it Coherent states and geometry on the Siegel--Jacobi
disk}, Int. J. Geom. Methods Mod. Phys.  {\bf{11} }
(2014) 1450035, 25 pages,  arXiv:1307.4219
\qquad\qquad\qquad

\bibitem{GAB} S. Berceanu,  {\it Geodesics associated to the balanced metric on
the Siegel-Jacobi ball},  Romanian J. Phys. {\bf  61}  (2016) 1137--1160;
 arXiv: 1605.02962v1 [math.DG]



\bibitem{SB15}S. Berceanu, {\it Balanced metric and Berezin quantization on  the
Siegel--Jacobi ball}, SIGMA {\bf 12} (2016) 064, 24 pages, arXiv:1512.00601

\bibitem{SB19a}S. Berceanu, {\it The real Jacobi group revisited},
 92 pages,  arXiv:1903.10721v1

\bibitem{SB19b}S. Berceanu, {\it The real Jacobi group revisited},
  SIGMA {{\bf 15}} (2019) 096, 50 pages; 54 pages,  arXiv:1903.10721v2 




\bibitem{sbl}S. Berceanu, L. Boutet de Monvel,
{\it  Linear dynamical systems, coherent
state manifolds, flows and matrix Riccati equation},  J. Math. Phys.
{\bf 34}  (1993), 2353--2371


\bibitem{BERC08B}  S. Berceanu,  A. Gheorghe, {\it Applications of the
Jacobi group to Quantum Mechanics},  Romanian J. Phys. {\bf 53} 
(2008) 1013--1021, arXiv:0812.0448


\bibitem{gem}S.  Berceanu, A.  Gheorghe,  {\it On the geometry of Siegel--Jacobi domains},
 Int. J. Geom. Methods Mod. Phys. {\bf 8}  (2011) 1783--1798;
 arXiv:1011.3317



\bibitem{SBS}S. Berceanu,  M. Schlichenmaier, {\it   Coherent state
    embeddings, polar divisors and Cauchy formulas},   J. Geom.
  Phys. {\bf 34}  (2000) 336--358, arXiv:9903105

%%%%%%%%%




\bibitem{ber73} F.A. Berezin, {\it Quantization in complex
bounded domains},  Dokl. Akad. Nauk
 SSSR \textbf{211} (1973) 1263--1266

\bibitem{ber74}F.A. Berezin, {\it  Quantization},
Math. USSR-Izv.  {\bf 38}  (1974) 1116--1175

\bibitem{ber75}F.A. Berezin, {\it  Quantization in complex symmetric
  spaces},   Math.
 USSR-Izv. {\bf 39} (1975)
363--402

\bibitem{berezin} F.A. Berezin,  {\it The general concept of
    quantization}, 
   Commun. Math. Phys. {\bf 40} (1975) 153--174


\bibitem{btv}J. Berndt, F.  Tricerri,  L.  Vanhecke, 
{\it Generalized Heisenberg groups and Damek-Ricci harmonic spaces}, 
Lecture Notes in Mathematics, Vol. 1598,  Springer-Verlag, Berlin, 1995 




   
\bibitem{bern84}R.  Berndt, {\it  Some differential operators in the
    theory of Jacobi forms}, preprint 
IHES/M/84/10,  1984, 31 pages


\bibitem{bern}R.  Berndt,  Sur l'arithmétique du corps des
fonctions elliptiques de niveau $N$, in {\it Seminar on number theory},
 {P}aris 1982--83 ({P}aris, 1982/1983),
 \textit{Progr. Math.}, Vol.~51, Birkh\"{a}user Boston, Boston, MA, 1984,
 21--32

\bibitem{bs}R. Berndt,  R. Schmidt, {\it Elements of the representation
theory of the Jacobi group}, Progress in Mathematics,  Vol.  163,  Birkh\"auser
Verlag, Basel, 1998



\bibitem{cart4}\'E. Cartan, 
 {\it La m\'ethode du rep\'ere mobile, la th\'eorie des groupes
   continus et les espaces g\'en\'eralis\'es}, \textit{Actualit\'es scientifiques et
 industrielles}, Vol. 194, Hermann \& Cie., Paris, 1935


 
\bibitem{cart5}\'E. Cartan, {\it Les espaces \`a connexion
    projective},
  Abh. Sem. Vektor -- Tensor analysis, Moskau \textbf{4} (1937) 147--173

\bibitem{chi}
 G. Chiribella, G.  Adesso, {\it  Quantum benchmarks for pure single-mode {G}aussian
 states}, Phys. Rev. Lett. \textbf{112} (2014) 010501, 6~pages;
 {arXiv:1308.2146}
 

\bibitem{don}S.  Donaldson,  {\it Scalar curvature and projective
    embeddings, I},   J. Diff. Geom. {\bf 59}   (2001) 479--522  


\bibitem{ez}M. Eichler,  D. Zagier, {\it The theory of Jacobi forms},
 Progress in
 Mathematics, Vol.~55, Birkh\"{a}user Boston, Inc., Boston, MA, 1985


\bibitem{ev} E.L. Evtushik (originator)  Moving-frame method, in {\it
    Encyclopedia of Mathematics}.  \newline
http://www.encyclopediaofmath.org/index.php?title=Moving-frame-method\&oldid=17828



\bibitem{gbt}
F. Gay-Balmaz, C. Tronci, {\it Vlasov moment f\/lows and geodesics on the {J}acobi
 group}, {J.~Math. Phys.} \textbf{53} (2012) 123502, 36~pages;
{arXiv:1105.1734}



\bibitem{green} R. Greene, Complex differential geometry, in {\it
    Differential geometry}, Lecture Notes in Mathematics, Vol. 1263,
  Editor V. L. Hansen,
 Springer -- Verlag, Berlin Heidelberg, 1987,   228--288


\bibitem{helg}  S. Helgason, {\it  Differential Geometry, Lie Groups and
Symmetric Spaces}, Pure and Applied Mathematics, Vol.~80, Academic Press, Inc., New
 York~-- London, 1978




\bibitem{cal3} E.  K\"ahler,
{\it Raum-Zeit-Individuum},
Rend. Accad. Naz. Sci. XL Mem. Mat. {\bf{16}}  (1992)
115--177 


\bibitem{cal} {\it Erich K\"ahler: Mathematische Werke/Mathematical
Works}, Editors  R. Berndt, O. Riemenschneider, Walter de Gruyter \&
 Co., Berlin, 2003
  


\bibitem{kn1} S. Kobayashi, K. Nomizu, {\it Foundations of Differential
Geometry}, {V}ol.~{I},
 Interscience Publishers, New York~-- London, 1963

\bibitem{kn} S. Kobayashi, K. Nomizu, {\it Foundations of Differential
Geometry}, {V}ol.~{II},
 Interscience Publishers, New York~-- London~-- Sydney, 1969



\bibitem{koda}T. Koda,  An introduction to the geometry of
    homogeneous spaces,  in {\it Proceedings
 of the 13th {I}nternational {W}orkshop on {D}ifferential {G}eometry and
 {R}elated {F}ields}, Natl. Inst. Math. Sci. (NIMS), Taej\u{o}n, 2009,
 121--144

\bibitem{kwv}O. Kowalski,  L. Vanhecke, {\it Riemannian manifolds
    with homogeneous geodesics},   Boll. Un. Mat. Ital.~B \textbf{5} (1991) 189--246

  
 \bibitem{KRSAR82} P. Kramer and  M. Saraceno, \textit{Semicoherent states and
the group} $\text{ISp}(2,\mathbb{R})$, Physics {\bf 114A} (1982)
448--453
\bibitem{LEE03} M.H. Lee, \textit{Theta functions on hermitian symmetric
domains and Fock representations}, J. Aust. Math. Soc. \textbf{74} (2003) 201--234


\bibitem{lis2}W. Lisiecki, {\it A classification of coherent state 
representations of  unimodular Lie groups}, Bull. Amer. Math. Soc. 
 {\bf 25} (1991) 37--43

 \bibitem{lis}W. Lisiecki, {\it Coherent state representations. A survey},
  Rep. Math. Phys. {\bf 35} (1995) 327--358


\bibitem{alo} A. Loi,  R. Mossa, {\it Berezin quantization of
    homogenous bounded domains}, Geom. Dedicata {\bf 161} (2012)
  119--128, arXiv:1106.2510

\bibitem{marmo}
G. Marmo, P.W. Michor,  Yu.A. Neretin, {\it The {L}agrangian {R}adon transform and the
 {W}eil representation}, J.~Fourier Anal. Appl. \textbf{20} (2014)
 321--361, arXiv:1212.4610

  
  
\bibitem{mosc}H. Moscovici, {\it Coherent state representations of
nilpotent Lie groups},  Commun. Math. Phys. {\bf 54} (1977)  63--68

\bibitem{mv}H. Moscovici,  A. Verona, {\it Coherent states and square integrable
representations},  Ann. Inst. Henri Poincar\'e {\bf 29}
(1978) 139--156





\bibitem{neeb96}K.-H. Neeb, {\it Coherent states, holomorphic
    extensions
    and highest weight representations},  Pacific J. Math. {\bf 174}
(1996) 230--261


\bibitem{neeb}K.-H. Neeb, {\it Holomorphy and Convexity in Lie
    Theory},   De Gruyter
 Expositions in Mathematics, Vol.~28, Walter de Gruyter \& Co., Berlin, 2000



\bibitem{nish}S. Nishiyama, J. Da Providencia, {\it Mean-field theory
    based on the $\got{Jacobi~hsp}$:= semidirect sum $\got{h}_N \rtimes
    \got{sp}(2N,\R)_{\C} $ algebra of boson operators},
  J. Math. Phys. {\bf{60}} (2019) 081706, 22 pages;  	arXiv:1809.01314 [hep-th]
 

 
\bibitem{nomizu}K. Nomizu, {\it Invariant affine connections on homogeneous
spaces},  Amer. J.  Math.  {\bf 76}   (1954)
 33--65



\bibitem{perG}A.M.   Perelomov,  {\it Generalized Coherent States and their
Applications}, Texts and
 Monographs in Physics, Springer-Verlag, Berlin, 1986

 
 
\bibitem{Q90} C. Quesne, \textit{Vector coherent state theory of the
semidirect sum Lie algebras wsp}$(2N,\mathbb{R})$, J. Phys. A:
Gen. {\mb 23}
(1990) 847--862



\bibitem{SA71} I. Satake, \textit{Fock representations and Theta Functions},
  Ann. Math. Studies \textbf{66} (1971) 393--405

\bibitem{SA71B} I. Satake, \textit{Unitary representations of a semi-direct
products of Lie groups on }$\bar{\partial}$\textit{-cohomology spaces},
Math. Ann.  \textbf{190} (1971) 177--202

\bibitem{SA71C} I. Satake, \textit{Factors of automorphy and Fock
representations}, Advances in Math. \textbf{7} (1971) 83-110 

\bibitem{SA80} I. Satake, \textit{Algebraic structures of symmetric
    domains}, Publ. Math. Soc. Japan 14, Princeton Univ. Press,  1980
  
\bibitem{SH03} K. Shuman, \textit{Complete signal processing bases and the
Jacobi group,} J. Math. Anal. Appl. 278 (2003), 203--213



\bibitem{tak}K. Takase, {\it  A note on automorphic forms}, J. Reine
  Angew. Math. {\bf 409} (1990) 138--171
  
\bibitem{TA92} K. Takase, \textit{On unitary representations of Jacobi
groups}, J. Reine Angew. Math. {\mb 430} (1992) 130--149

\bibitem{TA99} K. Takase, \textit{On Siegel modular forms of half-integral
weights and Jacobi forms}, Trans.  Amer. Math. Soc. 351  {\mb 2} (1999)
735--780 


\bibitem{tv}F. Tricerri,  L. Vanhecke,  
{\it Homogeneous structures on Riemannian manifolds},
 London Mathematical Society Lecture Note Series, Vol.~83, Cambridge
 University Press, Cambridge, 1983

 
\bibitem{AWeil}A. Weil, {\it Introduction \`a l'\'etude des vari\'et\'es
    k\"ahl\'eriennes}, Publications de l'Institut de Math\'ematique de
  l'Universit\'e  de Nancago, VI.  Actualit\'es scientifiques et industrielles,
  Vol. 1267,  Hermann, Paris,  1958 

\bibitem{Y02} J.-H. Yang, \textit{The method of orbits for real Lie groups},
Kyungpook Math. J. {\bf 42}  (2002) 199--272, arXiv:math.RT/060205


\bibitem{Yan} J.-H. Yang, {\it Remark on harmonic analysis on
    Siegel--Jacobi space}, arXiv: math.NT/0612230


\bibitem{Y07} J.-H. Yang, \textit{Invariant metrics and Laplacians on the
    Siegel--Jacobi spaces}, J.  Number Theory, {\bf 127} (2007) 83--102,
  arXiv:math.NT/0507215



\bibitem{Y08} J.-H. Yang, \textit{A partial Cayley transform for
    Siegel--Jacobi disk}, J. Korean Math. Soc.  {\bf 45}  (2008) 781--794,
  arXiv:math.NT/0507216

\bibitem{Y10} J.-H. Yang, \textit{Invariant metrics and Laplacians on the
    Siegel--Jacobi disk}, Chin. Ann. Math.  {\bf 31B}   (2010) 85--100,
  arXiv:math.NT/0507217




\end{thebibliography}
\end{document}